\documentclass[microtype]{gtpart}     % Basic GT/GTM/AGT style
\usepackage[utf8]{inputenc}
\usepackage{amsmath,amsthm,amsfonts,amssymb,verbatim,mathtools}
\usepackage{tikz-cd}
\usetikzlibrary{decorations.markings,arrows}
\tikzcdset{arrow style=tikz, diagrams={>=stealth}}
\usepackage{xcolor}
\usepackage{placeins}
\usepackage{pinlabel}

\usepackage{picins}
\usepackage[font=footnotesize,labelfont=bf]{caption}
\usepackage[font=footnotesize,labelfont=bf]{subcaption}
\usepackage{wrapfig}
\DeclareMathAlphabet{\mathpzc}{OT1}{pzc}{m}{it}

\definecolor{darkBlue}{RGB}{85,123,151}
\definecolor{lightBlue}{RGB}{114,178,223}
\definecolor{lightPeach}{RGB}{245,178,157}
\definecolor{darkPeach}{RGB}{226,103,67}
\definecolor{darkPeach}{RGB}{226,103,67}
\definecolor{lightGreen}{RGB}{82,157,54}
\definecolor{darkGreen}{RGB}{20,105,0}
\definecolor{lightBrown}{RGB}{190,166,159}

\def\co{\colon\thinspace}

\newcommand{\F}{\mathbb{F}}

\newcommand{\sA}{\mathcal{A}} 
\newcommand{\TsA}{\widetilde{\mathcal{A}}}

\newcommand{\sI}{\mathcal{I}}

\newcommand{\strip}{\mathbf{a}\times[-1,1]}

\newcommand{\HFhat}{\widehat{\mathit{HF}}}

\newcommand{\CFD}{\widehat{\mathit{CFD}}}

\newcommand{\Alg}{\mathcal{A}}

\newcommand{\id}{\operatorname{id}}

\usepackage{mathabx}

\newcommand{\gen}{\blackdiamond}
\newcommand{\specialgen}{\pluscirc}

\newcommand{\uvAlg}{{\mathcal R}}
\newcommand{\sampleAlg}{{\mathcal B}}
\renewcommand{\k}{{\mathbf k}}

\newcommand{\RHT}{T_{2,3}}
\newcommand{\bgamma}{\boldsymbol{\gamma}}
\newcommand{\openU}{\mathring{\nu}}

\newcommand{\handle}
	{\raisebox{-1.5pt}{\includegraphics[scale=0.5]{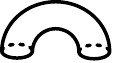}}}

\newcommand{\cut}
	{\raisebox{-1.5pt}{\includegraphics[scale=0.5]{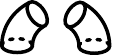}}}
	
\newtheorem{theorem}{Theorem}%[section]

\newtheorem{proposition}[theorem]{Proposition}

\newtheorem*{namedtheorem}{\theoremname}
\newcommand{\theoremname}{testing}

\theoremstyle{definition}

\newtheorem*{remarko}{Remark}
\newtheorem*{aside}{Aside}
\newtheorem*{example}{Example} 

\title[A mnemonic for the Lipshitz--Ozsv\'ath--Thurston correspondence]{A mnemonic for the\\ Lipshitz--Ozsv\'ath--Thurston correspondence}

\author{Artem Kotelskiy}
\givenname{Artem}
\surname{Kotelskiy}
\address{Department of Mathematics \\ Indiana University}
\email{artofkot@gmail.com}
\urladdr{https://artofkot.github.io/}

\author{Liam Watson}
\givenname{Liam}
\surname{Watson}
\address{Department of Mathematics \\ University of British Columbia}
\email{liam@math.ubc.ca}
\urladdr{https://personal.math.ubc.ca/~liam/}

\author{Claudius Zibrowius}
\givenname{Claudius}
\surname{Zibrowius}
\address{Department of Mathematics \\ University of Regensburg}
\email{claudius.zibrowius@posteo.net}
\urladdr{https://cbz20.raspberryip.com/}

\arxivreference{2005.02792}

\begin{document}

\begin{abstract} 
	%
	% arXiv version:
	%
	% When $\mathbf{k}$ is a field, type~D structures over the algebra $\mathbf{k}[u,v]/(uv)$ are equivalent to immersed curves decorated with local systems in the twice-punctured disk. Consequently, knot Floer homology, as a type~D structure over $\mathbf{k}[u,v]/(uv)$, can be viewed as a set of immersed curves. With this observation as a starting point, given a knot $K$ in $S^3$, we realize the immersed curve invariant $\widehat{\mathit{HF}}(S^3 \smallsetminus \mathring{\nu}(K))$ [arXiv:1604.03466] by converting the twice-punctured disk to a once-punctured torus via a handle attachment. This recovers a result of Lipshitz, Ozsv\'ath, and Thurston [arXiv:0810.0687] calculating the bordered invariant of $S^3 \smallsetminus \mathring{\nu}(K)$ in terms of the knot Floer homology of $K$.
	%
	%
	%
	%
When $\k$ is a field, type~D structures over the algebra $\k[u,v]/(uv)$ are equivalent to immersed curves decorated with local systems in the twice-punctured disk. Consequently, knot Floer homology, as a type~D structure over $\k[u,v]/(uv)$, can be viewed as a set of immersed curves. With this observation as a starting point, given a knot $K$ in $S^3$, we realize the immersed curve invariant $\widehat{\mathit{HF}}(S^3 \smallsetminus \openU(K))$~\cite{HRW} by converting the twice-punctured disk to a once-punctured torus via a handle attachment. This recovers a result of Lipshitz, Ozsv\'ath, and Thurston~\cite{LOT-main} calculating the bordered invariant of $S^3 \smallsetminus \openU(K)$ in terms of the knot Floer homology of $K$. 
\end{abstract}

\maketitle

\labellist \small
%\pinlabel $v$ at 30 9, \pinlabel $u$ at 140 9 \pinlabel $\mathbf{a}$ at 89 25
\pinlabel $v$ at 19 9, \pinlabel $u$ at 125 9 \pinlabel $\mathbf{a}$ at 64 25
\endlabellist
\parpic[r]{
 \begin{minipage}{45mm}
 \centering
 \captionsetup{type=figure}
 \includegraphics[scale=0.75]{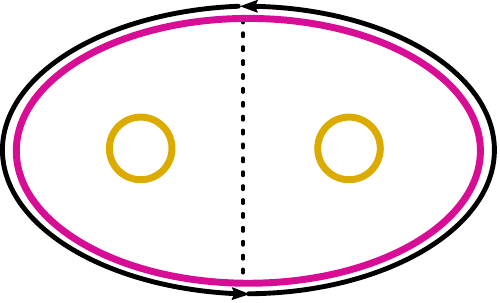}
 \captionof{figure}{An arc system associated with the algebra $\uvAlg$.}
 \label{fig:disk-with-quiver}
  \end{minipage}%
  }
Recent work interprets relative versions of homological invariants in terms of immersed curves, including Heegaard Floer homology for manifolds with torus boundary~\cite{HRW}, as well as link Floer homology~\cite{pqMod}, singular instanton knot homology~\cite{HHK1},  and Khovanov homology~\cite{KWZ} for 4-ended tangles. In particular~\cite[Section~5]{KWZ} classifies type~D structures over a quiver algebra associated with a surface with boundary in terms of immersed curves on this surface; compare~\cite{HKK,HRW}. Denoting a field by $\k$, perhaps the simplest algebra to illustrate these classification results is $\uvAlg=\k[u,v]/(uv)$. This algebra arises as the path algebra of a quiver that is associated with the decorated surface shown in Figure~\ref{fig:disk-with-quiver}. Work of Lekili and Polishchuk~\cite{LP-1,LP-2} describes the role of $\uvAlg$, and its relationship with the twice-punctured disk, in the context of homological mirror symmetry; see, in particular,~\cite[Figures~1 and~2]{LP-2}. And, the algebra $\uvAlg$ equipped with the Alexander and $\delta$ gradings $\text{gr}(u)=(-1,1),~\text{gr}(v)=(1,1)$ plays a central role in knot Floer homology; see~\cite{DHST}, for instance.

\bigskip

\medskip

\begin{theorem}\label{thm:classification}
Every bigraded type~D structure over $\uvAlg$ is equivalent to an immersed curve (decorated with local systems) in the twice-punctured disk, which is unique up to regular homotopy (and equivalence of local systems). 

% The bigrading is needed for the step from precurves to simply-faced precurves: First, we need the homological or quantum (but not δ-) grading to assume the f-join connects different generators. Then, to apply the basic isomorphism, we need h²=0, for which we may use the δ- or quantum, but not homological grading. The homological or quantum grading is needed again to show that all arrows can be homotoped away except for local systems.

\end{theorem}

As stated, this is a special case of a theorem proved in~\cite[Section~5]{KWZ} appealing to techniques from~\cite{HRW} (see also~\cite{pqMod}). The observation could alternatively be extracted from~\cite[Section~3.4]{HRW} (see the remark in Section~\ref{sub:geo} accompanying Figure~\ref{fig:cut-annulus}), and also follows from work of Haiden, Katzarkov, and Kontsevich~\cite{HKK}; see \cite[Section~1.8]{KWZ} for more discussion. 
We will review the algebraic objects in Section~\ref{sub:alg} and, without reproducing the proof in full, explain some key steps in this special case in Section~\ref{sub:geo}. 
%Theorem~\ref{thm:classification} gives rise to a graphical interpretation $\bgamma$ of ${}^\uvAlg\mathit{CFK}(Y,K)$, a variant of knot Floer homology which is a bigraded type~D structure over $\uvAlg$. 
Theorem~\ref{thm:classification} gives rise to a graphical interpretation $\bgamma$ for (a variant of) knot Floer homology ${}^\uvAlg\mathit{CFK}(Y,K)$, which is a bigraded type~D structure over $\uvAlg$. 
Our proof is constructive and, in particular, foregrounds the role of vertically and horizontally simplified bases that arise in knot Floer homology. An explicit example of a curve~$\bgamma$ in the twice-punctured disk is shown on the left of Figure~\ref{fig:h}. This particular curve corresponds to the type~D structure associated with the right-hand trefoil $\RHT$  in $S^3$: 
\begin{equation*}\label{RHT}[ \gen_1 \xleftarrow{u}\gen_2 \xrightarrow{v}\gen_3] = {}^\uvAlg\mathit{CFK}(S^3, \RHT)\end{equation*}
Note that, while the local system in this example is trivial, these are easy to add to the picture in general, being equivalent to isomorphism classes of flat vector bundles over the curves in question. There is an obvious handle attachment, identifying the two punctures in the disk, which yields a once-punctured torus. Denote this handle attachment by $\handle$ and consider the curve $\handle({\bgamma})$. Note that, given a choice of meridian on the torus, this operation has an inverse that we will denote by~$\cut$. 

Denote by $\widehat{\mathit{HF}}(M)$ the immersed curve in the once-punctured torus associated with a manifold $M$ with torus boundary~\cite{HRW}. This is equivalent to the bordered Heegaard Floer invariant of $M$~\cite{LOT-main}. Here is the mnemonic we propose:

\labellist \small
\pinlabel $\mathbf{a}$ at 77 19
\pinlabel $\mathbf{a}_\bullet$ at 297 24
\pinlabel $\mathbf{a}_\circ$ at 340 35
%\pinlabel $\mathbf{h}$ at 195 22
\endlabellist
\begin{figure}[t]
 \includegraphics[width=\textwidth]{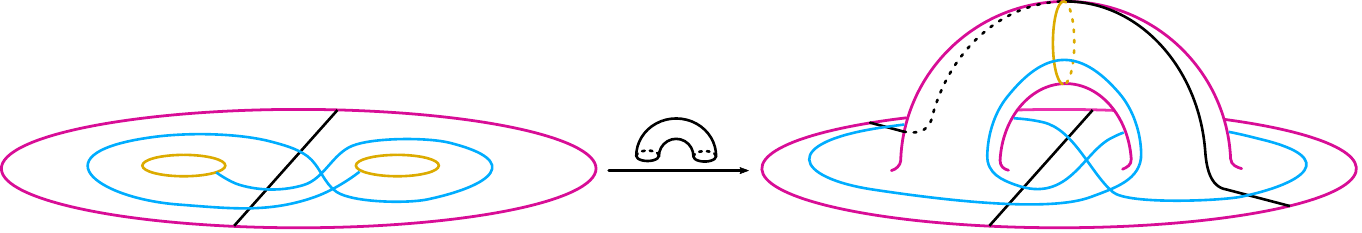}
 \caption{Adding a handle to the twice-punctured disk results in the once-punctured torus. This carries immersed curves to immersed curves; the immersed curve on the left corresponds to the type~D structure ${}^\uvAlg\mathit{CFK}(S^3, \RHT)=[\gen_1 \xleftarrow{u}\gen_2 \xrightarrow{v}\gen_3]$, which is carried to the curve $\widehat{\mathit{HF}}(M)$ where $M$ is the trefoil exterior $S^3\smallsetminus\openU(K)$.}\label{fig:h}
\end{figure}

\begin{theorem}\label{thm:LOT}
If $\bgamma$ is a curve representing the knot Floer invariant ${}^\uvAlg\mathit{CFK}(S^3,K)$ over the two-element field, then $\handle({\bgamma})$ is equivalent to $\widehat{\mathit{HF}}(M)$, where $M$ is the exterior of the knot $K$. Conversely, given a meridian for $M=S^3\smallsetminus\openU(K)$, the curve $\cut\big(\HFhat(M)\big)$ represents the knot Floer type~D structure for $K$. 
\end{theorem}
Figure~\ref{fig:h} illustrates this theorem for the right-hand trefoil knot; the proof is given in Section~\ref{sub:handle}. 

\begin{remarko}
There is an apparent ambiguity in the statement of Theorem~\ref{thm:LOT}, namely the number of twists (along the belt of the handle $\handle$) one adds to the non-compact component of the curve $\bgamma$. However, recall that the curve $\widehat{\mathit{HF}}(M) \subset \partial M$ is null-homologous in $M$~\cite[Sections~5 and~6]{HRW}; to resolve the ambiguity it is enough to identify the once-punctured torus obtained after adding the handle with the boundary of the knot exterior (minus a small disk). We identify the arc $\mathbf{a}_\bullet$ from Figure~\ref{fig:h} with the meridian~$\mu$, and the second arc $\mathbf{a}_\circ$ with a longitude $\lambda$ of $K$. This pair provides a bordered structure, in the sense of Lipshitz--Ozsv\'ath--Thurston~\cite{LOT-main}. Concerning the framing~$\lambda$: On one hand there is a preferred choice given by the Seifert longitude $\lambda_0$, and the corresponding identification is depicted on the right of Figure~\ref{fig:framings}. On the other hand, it is often simplest to work with the ``blackboard framing'', which simply joins the endpoints of $\bgamma$ without new twisting as they run over the handle, as in Figure~\ref{fig:h}. In general, this latter gives the $2\tau(K)$-framed longitude $\lambda_{2\tau}=2\tau\cdot\mu+\lambda_0$, where the value $\tau(K)$ is the Ozsv\'ath--Szab\'o concordance invariant (we describe how to extract this value below). This choice of longitude is illustrated on the left in Figure~\ref{fig:framings}. These choices differ by Dehn twists along~$\mu$; note that in both cases $[\handle(\bgamma)] = [\lambda_0]$ in homology. Different choices of twisting precisely correspond to different unstable chains appearing in~\cite[Theorem~A.11]{LOT-main}, due to Lipshitz, Ozsv\'ath, and Thurston, which Theorem~\ref{thm:LOT} re-casts. This result generalizes to knots in arbitrary three-manifolds; see Section~\ref{sub:discussion} for further discussion. 
\end{remarko}

\labellist \small
%\pinlabel $\mathbf{a}$ at 77 19
%\pinlabel $\mathbf{a}$ at 298 23
%\pinlabel $\mathbf{b}$ at 340 35
%\pinlabel $\mathbf{h}$ at 195 22
\pinlabel $\mu$ at -4.5 56 \pinlabel $\lambda_{2\tau}$ at 176 56
\pinlabel $\mu$ at 230.5 56 \pinlabel $\lambda_{0}$ at 415 47
\endlabellist
\begin{figure}[t]
	\centering
 \includegraphics[width=0.96\textwidth]{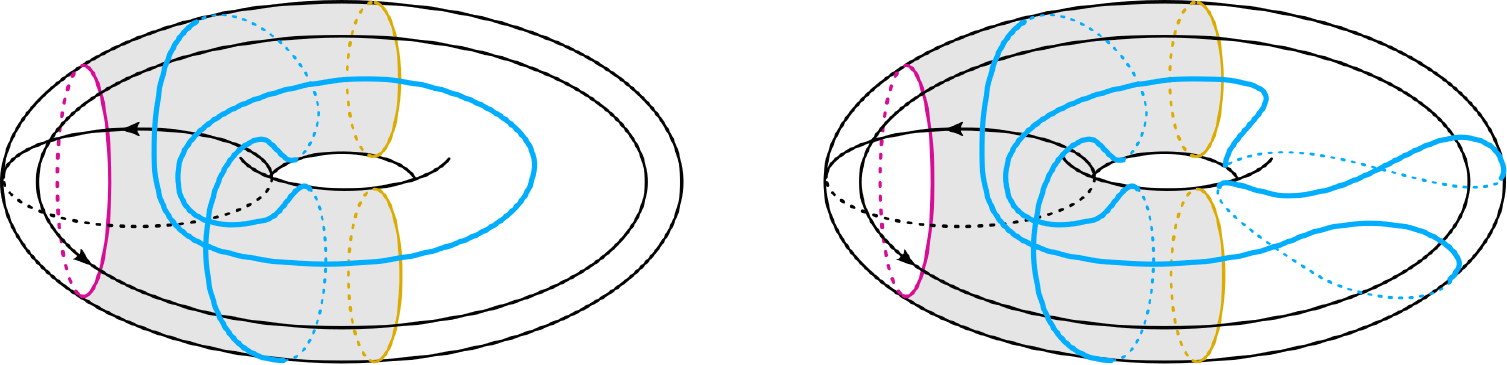}
 \caption{Choices of framing on the right-hand trefoil invariant: $\lambda_{2\tau}=2\mu+\lambda_0$ on the left and the Seifert longitude $\lambda_0$ on the right. The resulting curve $\gamma$ on the boundary of the trefoil exterior coincides with~\cite[Figure~9]{HRW-prop}.}\label{fig:framings}
\end{figure}

A graphical interpretation of the family of concordance homomorphism $\{\phi_i\}$ due to Dai, Hom, Stoffregen, and Truong~\cite{DHST} is given by Hanselman and the second author~\cite{HW}. This can be read off the current picture: Denote by $\bgamma_0(K)\subset \bgamma(K)$ the non-compact curve in the twice-punctured disk associated with ${}^\uvAlg\mathit{CFK}(S^3,K)$. (The curve $\bgamma_0(K)$ is a concordance invariant~\cite[Proposition 2]{HW}.) Orient $\bgamma_0(K)$ so that it leaves from the $v$-puncture; this is the left-hand puncture in Figure~\ref{fig:disk-with-quiver}, which records the $v^i$ coefficient maps. Contracting the arc $\mathbf{a}$ to a point gives a wedge of annuli $A_v\vee A_u$, and the oriented segments of $\bgamma_0(K)$ around the $v$-puncture give a collection of homotopy classes in $\pi_1A_v\cong\langle t \rangle$, where the generator $t$ winds counterclockwise. As a result, given $\bgamma_0(K)$ with our choice of orientation we obtain  $t^{n_1}t^{n_2}\cdots t^{n_k}$ for the $k$ oriented segments winding around the $v$-puncture, and $$\phi_i(K)=\sum_{n_j = \pm i} \text{sign}(n_j) \qquad  \qquad 
\tau(K)=\sum_{j=1}^k n_j
$$
so that $\tau(K)$ is simply the winding number of $\bgamma$ around the $v$-puncture. One can check that this gives $\tau(\RHT)=\phi_1(\RHT)=1$. A more complicated example is shown in Figure~\ref{fig:cable}. The same construction works with the $u$-puncture instead of the $v$-puncture, due to a symmetry interchanging $u$ and $v$ in knot Floer homology~\cite{OS-hfk-0}.

\labellist 
\pinlabel $v$ at 9 50 \pinlabel $v^2$ at 56 50 \pinlabel $v$ at 96 50
\pinlabel $u$ at 417 27 \pinlabel $u^2$ at 417 74 \pinlabel $u$ at 417 109
\endlabellist
\begin{figure}[t]
 \includegraphics[width=\textwidth]{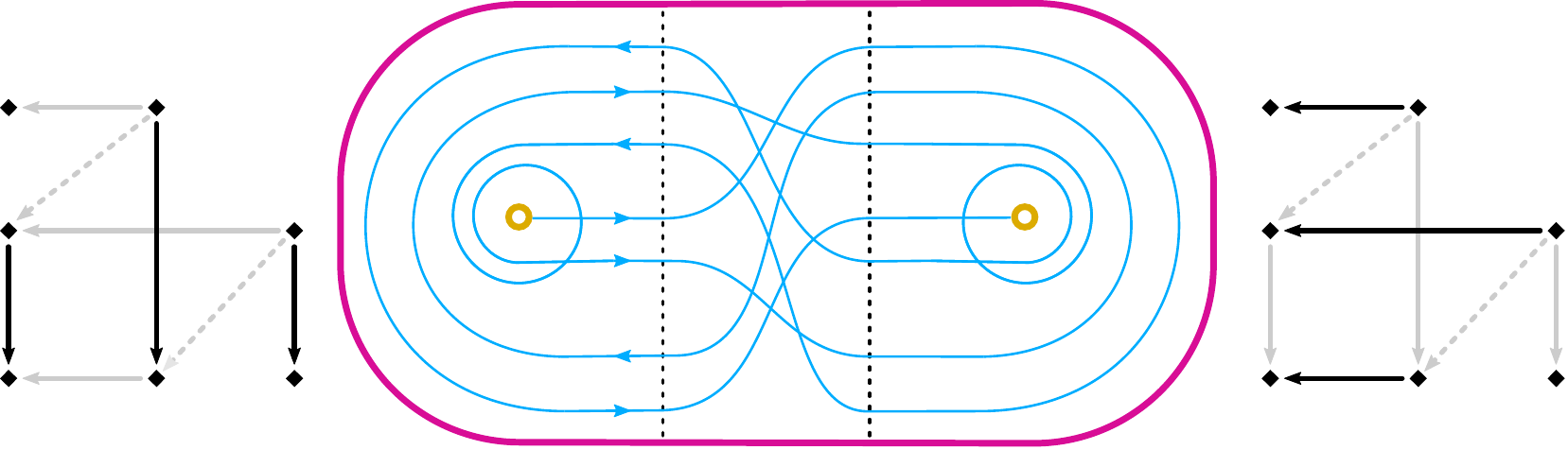}
% \vspace{0.5cm}
 \caption{The curve associated with $ {}^\uvAlg\mathit{CFK}(S^3,K)$ when $K$ is the $(2,1)$-cable of the right-hand trefoil. The vertical and horizontal complexes are shown beside the relevant annuli; including the diagonal arrows describes the invariant over $\mathbf{k}[u,v]$. Applying Theorem~\ref{thm:LOT} results in the curve-invariant in the torus, which can be compared with~\cite[Figure~1]{HW}. Orientating the curve as shown, we calculate $\phi_1(K)=0$, $\phi_2(K)=1$, and $\tau(K)=2$. 
 }\label{fig:cable}
\end{figure}

Relevant to concordance is the behaviour under connect sum. Denote by $_{\uvAlg}\mathit{HFK}(S^3,K)$ the knot Floer invariant obtained as the homology of a complex  $\mathit{CFK}(S^3,K)$ freely generated over $\uvAlg$. In Section~\ref{sub:sum} we prove: 
\begin{theorem}\label{thm:connected_sum}
  The knot Floer homology over $\uvAlg$ of a connected sum of two knots is equal to the wrapped Lagrangian Floer homology of the corresponding curves:
  \[ {}_{\uvAlg}\mathit{HFK}(S^3, \text{m} K \# K') \cong \mathit{HF} (\bgamma(K), \bgamma(K')) \]
\end{theorem}
A proof is given in Section \ref{sub:sum}. As is the case with Theorem~\ref{thm:classification}, the proof 
%follows from
appeals to the techniques in~\cite[Section~5]{KWZ}.

\section{Algebraic objects}\label{sub:alg} Let $\sampleAlg$ be a bigraded unital algebra over a field $\k$, with a subring of idempotents $\mathcal{I}$ being equal to $\k^n$. The object of interest is a bigraded chain complex over $\sampleAlg$: Let $V$ be a finite dimensional bigraded left $\mathcal{I}$-module, and suppose further that we have a morphism of  $\mathcal{I}$-modules
\[d\co V \to \sampleAlg\otimes_{\mathcal{I}} V\]
%of bidegree $a^0\delta^1$, 
satisfying the compatibility condition
\[(\mu \otimes \id_V)\circ ( \id_\sampleAlg \otimes d ) \circ d = 0\]
where $\mu$ denotes multiplication in $\sampleAlg$. 
In our setting the morphism $d$ has bidegree $(a,\delta)=(0,1)$,
and the pair $(V,d)$ is a bigraded type~D structure over $\sampleAlg$.

A couple of remarks: We work with left actions for consistency with~\cite{LOT-main}, and our type~D structures will always be reduced, which means that $d(x)=\sum_i b_i\otimes y_i$ where none of the $b_i\in\sampleAlg$ are invertible. This is justified by the fact that any bigraded type~D structure is homotopy equivalent to a reduced one~\cite[Lemma~2.16]{KWZ}.

Such algebraic structures appear naturally in a variety of settings. For example, given a knot $K$ in $S^3$, the knot Floer invariant $\mathit{HFK}(S^3,K)$, due to Ozsv\'ath--Szab\'o~\cite{OS-hfk-0} and to Rasmussen~\cite{Rasm_hfk}, can be viewed as a ${\k[u,v]}$-module obtained as the homology of a chain complex $\mathit{CFK}(S^3,K)$ over the ring $\k[u,v]$~\cite[Section~3]{Zemke1}. This complex is freely generated as a module over this ring. As such, it is natural to view $\mathit{CFK}(S^3,K)$ as a type~D structure over $\k[u,v]$, which we denote by  ${}^{\k[u,v]}\mathit{CFK}(S^3,K)$. 

Given a type~D structure over $\sampleAlg$, a homomorphism of $\mathcal I$-algebras $\sampleAlg\to \sampleAlg'$ gives rise to an induced type~D structure over $\sampleAlg'$. In particular, the quotient $\k[u,v]\to\k[u,v]/(uv)$  defines a truncated version of the knot Floer type~D structure: 
\[ {}^\uvAlg\mathit{CFK}(S^3,K)=  {}^{\k[u,v]}\mathit{CFK}(S^3,K)\big|_{uv=0}\]
The associated module object ${}_\uvAlg\mathit{CFK}(S^3,K)$ (see~\cite[Lemma~2.20]{LOT-main}) is the knot Floer complex freely generated over $\uvAlg$, which is studied in depth by Dai, Hom, Stoffregen, and Truong~\cite{DHST} and Ozsváth and Szabó~\cite{OS_latest}. A concise formula connecting the type~D structure and the associated module object uses the box tensor product (see \cite[Section~2.3.2 and Proposition~2.3.18]{LOT-bim}, and also the beginning of Section~\ref{sub:handle} for a similar construction):
$${}_\uvAlg\mathit{CFK}(S^3,K)={}_{\uvAlg}\uvAlg_{\uvAlg} \boxtimes {}^\uvAlg\mathit{CFK}(S^3,K)$$
We note that there are two further type~D structures obtained from ${}^\uvAlg\mathit{CFK}(S^3,K)$ by setting the appropriate variables equal to zero: 
the horizontal type~D structure $C^{\mathbf{h}}$ and the vertical type~D structure $C^{\mathbf{v}}$. For instance, in the case of the type~D structure ${}^\uvAlg\mathit{CFK}(S^3,T_{2,3})$ (see Figure~\ref{fig:h}), we have 
$$C^{\mathbf{h}}=[\gen_1 \xleftarrow{u}\gen_2  \quad   \gen_3] \qquad \qquad C^{\mathbf{v}}= [\gen_1 \quad \gen_2  \xrightarrow{v} \gen_3]$$ 
As the type~D structures are reduced, the isomorphisms of vector spaces 
$C^{\mathbf{h}}\big|_{u=0} \cong \widehat{\mathit{HFK}}(S^3,K) \cong C^{\mathbf{v}}\big|_{v=0}$  induce an isomorphism
\[\varphi\co {}C^{\mathbf{h}}\big|_{u=0} \to C^{\mathbf{v}}\big|_{v=0}\]
 We have:
\begin{proposition}\label{prp:vert-hor}The data specified by the triple $({}C^{\mathbf{h}}, ~ C^{\mathbf{v}}, ~\varphi)$ is equivalent to the type~D structure ${}^{\uvAlg}\mathit{CFK}(S^3,K)$.
\end{proposition}
\begin{proof} This is immediate from the definitions, but also follows from the discussion in Section~\ref{sub:geo} outlining the proof of Theorem~\ref{thm:classification}. \end{proof}

% Matt Hedden's comment: 
% (0) Maybe add a preliminary paragraph to section 2 explaining the basic idea of HKK (or, as interpreted by HRW, KWZ), and how it will be invoked in this particular example;

\section{Geometric objects}\label{sub:geo}
Often, when an invariant of a topological object is a type~D structure over an algebra~$\sampleAlg$, the invariant is only well-defined up to homotopy equivalence. As such, it is of general interest to be able to classify homotopy equivalence classes of type~D structures. Such classification turns out to be possible when the algebra $\sampleAlg$ is isomorphic to an endomorphism algebra of certain objects in the (wrapped) Fukaya category of a surface~$\Sigma$. 
In this case, homotopy equivalence classes of type~D structures over $\sampleAlg$ correspond to certain curves (decorated with local systems) immersed in~$\Sigma$. 
This is a powerful structural result allowing us to translate algebra into geometry, 
%which is remarkable because it is much more common in mathematics to translate geometric problems into algebra. 
% This is a powerful structural result enabling transition from algebra to geometry,
something not so often encountered in mathematics.
%The classification result was first proven in~\cite{HKK} using the representation theory of nets, and later also reproved in~\cite{KWZ,pqMod,HRW} using the more geometric approach via train tracks/precurves and appealing to a certain simplification algorithm originally developed in~\cite{HRW}. We focus on the second approach.
The classification result is established in~\cite{HKK} using representations of nets; an alternate, more geometric approach is given in~\cite{HRW}, which appeals to train tracks in a surface. The simplification algorithm proved in \cite{HRW} that is central to the classification is further developed and leveraged in~\cite{KWZ,pqMod}, where train tracks reappear as precurves. We focus on this latter approach.

\labellist \small
\pinlabel $v$ at 19 40, \pinlabel $u$ at 115 40
\endlabellist
\parpic[r]{
	\begin{minipage}{45mm}
		\centering
		\captionsetup{type=figure}
		\includegraphics[scale=0.75]{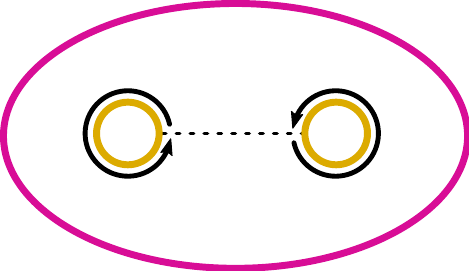}
		\captionof{figure}{A dual arc system associated with the algebra $\uvAlg$.}
		\label{fig:disk-with-quiver-dual}
	\end{minipage}%
}
To provide a useful toy model for the classification result, we restrict to type~D structures over~$\uvAlg$.
The algebra $\uvAlg$ indeed arises as the endomorphism algebra of an object in the (wrapped) Fukaya category of a surface. 
The surface is the oriented, twice punctured disk $D$ and the object is an arc connecting the two punctures; see Figure~\ref{fig:disk-with-quiver-dual}. 
More explicitly, from this figure we can extract a quiver with a single vertex  corresponding to the object in the Fukaya category, and arrows labeled $u$ and $v$ corresponding to the two paths around the punctures in $D$:
\[
\setlength\mathsurround{0pt}
% to avoid rendering issue with gtpart.cls
% see https://tex.stackexchange.com/questions/217753/rendering-issue-with-tikz-cd-package-and-gtpart-cls-document-class
v
\begin{tikzcd}[row sep=2cm, column sep=1.5cm]
\gen
\arrow[leftarrow,in=145, out=-145,looseness=5]{rl}{}
\arrow[leftarrow,out=35, in=-35,looseness=5]{rl}{}
\end{tikzcd}
u
\]
It is useful to view this quiver as a deformation retract of the twice-punctured disk. The algebra $\uvAlg$ is the path algebra of this quiver modulo the relations $uv=0=vu$. In terms of Figure~\ref{fig:disk-with-quiver-dual}, these relations have the effect that paths that run along the dashed arc are zero in $\uvAlg$, while paths that only wind around a single puncture are non-zero. 

To match the setup in~\cite{KWZ} a different viewpoint, which is in some sense dual to the previous one, will be more useful. Namely, choose an arc $\mathbf{a}$ that is properly embedded in $(D,\partial D)$ and that divides $D$ into a pair of annuli, as illustrated in Figure~\ref{fig:disk-with-quiver}. From this, we can also recover the quiver: The vertex corresponds to the arc $\mathbf{a}$ and the arrows correspond to paths on the boundary of $D$. Again, it is useful to consider the quiver as a deformation retract that contracts the arc to the quiver vertex. The relations that we impose on the quiver algebra to obtain $\uvAlg$ now have a different geometric interpretation: Paths that at an endpoint of the dashed arc continue along the boundary of $D$  are zero in~$\uvAlg$, while paths that at such a point always choose to follow the dashed arc are non-zero; see also~\cite[Section~5.1]{KWZ}.

The choice of arc $\mathbf{a}$ is an example of an arc system on \(D\), in the sense of~\cite[Section~5.1]{KWZ}.  In general, an arc system, giving rise to an algebra $\sampleAlg$, allows for a graphical representation of type~D structures over $\sampleAlg$ as sub-objects of the surface. These show up as train tracks in~\cite{HRW} and precurves in~\cite{KWZ}; we describe them explicitly in the case of $\uvAlg$ and the twice-punctured disk $D$. It will be convenient to specify the annuli $D\smallsetminus \mathbf{a}=A_v \sqcup A_u$; these annuli are called faces.

Let $(V,d)$ be a type~D structure over $\uvAlg$. Given a homogeneous basis $\{x_1,\ldots,x_n\}$ for~$V$ (as a vector space over $\k$, say), we can pick $n$ distinct points on $\mathbf{a}$ and label these with the $x_i$.  To describe the morphism $d$, suppose $b\otimes x_j$ is a summand of $d(x_i)$. Then, since $b$ is a sum of polynomials, we may assume without loss of generality that $b$ is $\lambda u^k$ or $\lambda v^k$ for some $\lambda \in \k$ and $k>0$. (The assumption that this power is non-zero comes from our restriction to reduced type~D structures.) There are two cases: if $b=\lambda u^k$ then we connect $x_i$ to $x_j$ by an oriented arc immersed in $A_u$ that winds algebraically $k$ times in the positive direction; and if $b=\lambda v^k$ then we connect $x_i$ to $x_j$ by an oriented curve immersed in $A_v$ that winds algebraically $k$ times in the positive direction. In both cases the arc is decorated by the field coefficient $\lambda$, noting that when $\lambda=1$ our convention is to drop the label. In particular, when $\k$ is the two-element field, only the arcs are needed. Lastly, if an intersection point $x_i$ does not have outgoing arcs in the annulus~$A_u$, we connect $x_i$ straight to the $u$-puncture; we do the same for the $A_v$ annulus and the $v$-puncture. To see that this information, having added all of the arcs described, can be viewed as an immersed train track in $D$, we simply require that every curve is perpendicular to $\mathbf{a}$ in a neighborhood of each $x_i$. An explicit example is given in Figure~\ref{fig:run-exp-1}. Note that in this example there are no arcs going to interior punctures.

\labellist %\small
\pinlabel $A_v$ at 80 25
\pinlabel $A_u$ at 180 25
\tiny
\pinlabel $-\!1$ at 95 97
\endlabellist
\begin{figure}[ht]
	\centering
 \includegraphics[scale=1.1]{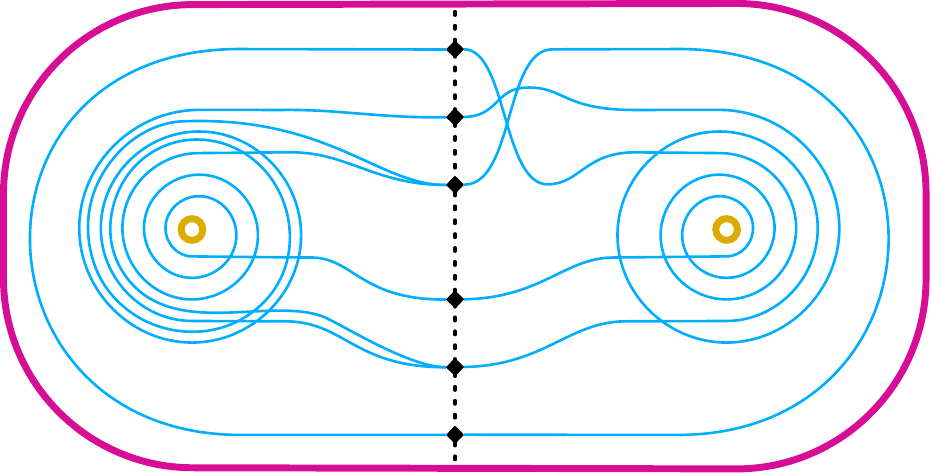}
 \caption{A sample train track representation of a type~D structure over $\uvAlg$. Note that every curve segment is oriented so that it runs counter-clockwise around a puncture, so this orientation is omitted. Similarly, unlabeled edges (of which there are all but one in this example) carry the decoration $\lambda=1$.}\label{fig:run-exp-1}
\end{figure}

These train tracks can be put into a simple form that makes them easier to manage: We require that they are simply faced in the sense of~\cite[Definition 5.9]{KWZ}. In the present setting, this amounts to expressing 
\[D=A_v\cup_{\mathbf{a}\times\{1\}}\big(\mathbf{a}\times[-1,1]\big)\cup_{\mathbf{a}\times\{-1\}} A_u\]
and requiring that the train track restricted to $A_u$ and to $A_v$ describes a type~D structure over $\k[u]$ and $\k[v]$, respectively, with the property that each $x_i$ connects to at most one $x_j$. For an illustration see Figure~\ref{fig:run-exp-2}. All of the interesting switching is confined to the strip $\mathbf{a}\times[-1,1]$, which amounts to a graphical interpretation (reading from right to left) of an isomorphism $\varphi\co V_u\to V_v$, where $V_v$ and $V_u$ are the underlying vector spaces associated with the type~D structure in each face. As such, the general fact that we can restrict to simply faced train tracks (see~\cite[Proposition~5.10]{KWZ}) boils down to the fact that type~D structures over $\uvAlg$ admit vertically and horizontally simplified bases~\cite[Definition~11.23]{LOT-main}---though not necessarily one that is simultaneously vertically and horizontally simplified, whence the choice of isomorphism.  This last assertion explains the presence of $\varphi$; compare Proposition~\ref{prp:vert-hor}. We remark that this is one step in which the grading plays a key role. 

\labellist %\small
\pinlabel $\underbracket{\phantom{aaaaaaaaaa}}$ at 142 0 
\pinlabel $\mathbf{a}\times[-1,1]$ at 142 -12
\pinlabel $\overbrace{\phantom{aaaaaaaaaa}}$ at 142 138
\pinlabel {$\protect{\left(\begin{smallmatrix}0&0&1\\ \!\!\!\!-\!1&1&0\\1&0&0 \end{smallmatrix}\right)}$} at 142 154
%\pinlabel {$\protect{\left(\begin{smallmatrix}0&0&1\\-1&1&0\\1&0&0 \end{smallmatrix}\right)}$} at 300 22
\tiny
\pinlabel $-\!1$ at 123 91 \pinlabel $1$ at 134 91.5
\endlabellist
\begin{figure}[ht]
\vspace{1cm}
\centering
 \includegraphics[scale=1.1]{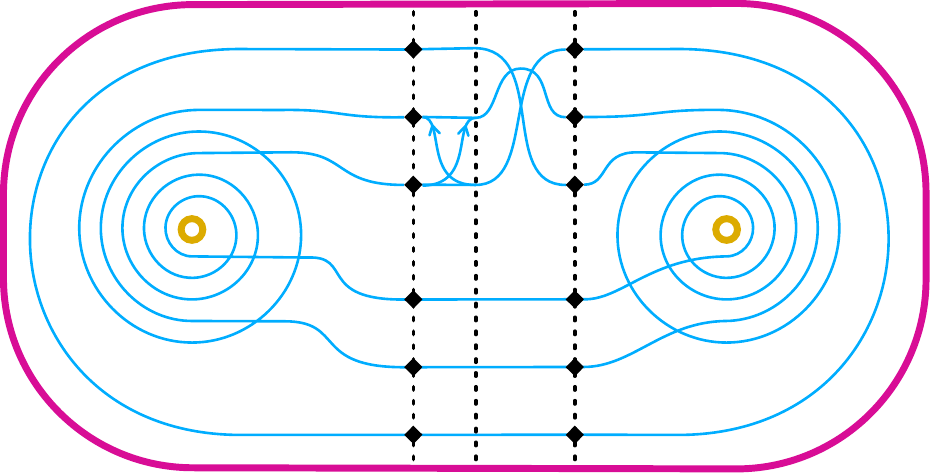}
 \vspace{0.5cm}
 \caption{Expressing the train track from Figure~\ref{fig:run-exp-1} as a simply faced precurve. 
 The isomorphism described can be read off the tracks in $\mathbf{a}\times[-1,1]$ from right to left; in the present setting the resulting matrix block-decomposes into two $3\times 3$ parts, of which one is shown and the other is the identity matrix.
 }\label{fig:run-exp-2}
\end{figure}

% Matt Hedden's comment:
% (1) Expand the explanation given in the remark on the border of pgs 6-7; in particular, unpack for the uninitiated "the key observation...";

\begin{aside} 
		We make a digression to describe that, in order to classify type~D structures in terms of immersed curves, other choices of surface decomposition are possible. Namely, another option would be to (1) cut the annuli $A_u$ and $A_v$ further, as described in Figure~\ref{fig:cut-annulus}; (2)~associate with this new geometric picture a different algebra $\mathcal{E}$; (3) interpret the type~D structure ${}^\uvAlg V$ as a type~D structure ${}^\mathcal{E} W$ over the algebra $\mathcal{E}$; (4) apply the methods from~\cite{HRW} to interpret ${}^\mathcal{E} W$ as an immersed curve. To describe this in more detail, let us focus first on the annulus $A_v$ in step~(2).

		Consider Figure~\ref{fig:cut-annulus}.
		Any type~D structure ${}^{\k[v]}V_{\gen}$ may be regarded as a type~D structure 
		$V_{\gen}\oplus V_\bullet\oplus V_\circ$ over the quiver algebra $\k[\bullet \xrightarrow{a} \gen \xrightarrow{b} \circ]$  together with an isomorphism  between the vector spaces $V_\bullet$ and $V_\circ$. To repackage the latter into a type~D structure without extra data, we consider a subalgebra generated by idempotents $\iota_\bullet+\iota_\circ$ and $\iota_\gen$ (because eventually  the  idempotents $\iota_\circ$ and $\iota_\bullet$ are identified). 
		Writing $\iota_\specialgen=\iota_\bullet+\iota_\circ$, the  subalgebra is equal to \[\mathcal{C}= \k[\specialgen \overset{a}{\underset{b}\rightleftarrows} \gen]/(ba)\] 
		The type~D structure ${}^{\k[v]}V_{\gen}$ can now be interpreted as a type~D structure 
		${}^{\mathcal{C}}(V_{\gen}\oplus V_{\specialgen})$: Generators $\gen$ in ${}^{\k[v]}V_{\gen}$ and ${}^{\mathcal{C}}(V_{\gen}\oplus V_{\specialgen})$ are in one-to-one correspondence, while a differential $\gen \xrightarrow{v^n} \gen$ in ${}^{\k[v]}V_{\gen}$  corresponds to the sequence of differentials 
		$$\gen \xrightarrow{b}\underbrace{\specialgen \xrightarrow{ab} \specialgen \xrightarrow{ab} \cdots \xrightarrow{ab} \specialgen}_{n \text{ generators}} \xrightarrow{a} \gen $$
		in ${}^{\mathcal{C}}(V_{\gen}\oplus V_{\specialgen})$. 
		To add the second annulus $A_u$ to the picture,  given a type~D structure ${}^{\uvAlg}V_{\gen}$ one translates it into a type~D structure ${}^\mathcal{E}W$ over the algebra
		$$\mathcal{E}= \k[\specialgen_1 \overset{a_1}{\underset{b_1}\rightleftarrows} \gen \overset{b_2}{\underset{a_2}\rightleftarrows} \specialgen_2 ]/(b_1a_1,b_2 a_2,a_1b_2,a_2b_2)$$
		via the dictionary
		\begin{align}
		\gen \xrightarrow{v^n} \gen&\quad \mapsto \quad \gen \xrightarrow{b_1}\underbrace{\specialgen_1 \xrightarrow{a_1b_1} \specialgen_1 \xrightarrow{a_1b_1} \cdots \xrightarrow{a_1b_1} \specialgen_1}_{n \text{ generators}} \xrightarrow{a_1} \gen \label{ping1} \\
		\gen \xrightarrow{u^n} \gen& \quad\mapsto \quad \gen \xrightarrow{b_2}\underbrace{\specialgen_2 \xrightarrow{a_2b_2} \specialgen_2 \xrightarrow{a_2b_2} \cdots \xrightarrow{a_2b_2} \specialgen_2}_{n \text{- generators}} \xrightarrow{a_2} \gen \label{ping2} 
    \end{align}
		With this type~D structure ${}^\mathcal{E}W$ in hand, the methods from~\cite{HRW} allow us to interpret ${}^\mathcal{E} W$ as an immersed curve. 

    A possible difficulty might arise from the following: The passage from ${}^\uvAlg V_{\gen}$ to ${}^\mathcal{E}W$ does not respect homotopy equivalences: There exist homotopy equivalent type~D structures ${}^\uvAlg V_{\gen} \simeq {}^\uvAlg V'_{\gen}$ such that the corresponding type~D structures ${}^\mathcal{E}W$ and ${}^\mathcal{E}W'$ are not homotopy equivalent (take, for example, ${}^\uvAlg V_{\gen}=[\gen\xleftarrow{v}\gen\xrightarrow{v}\gen]$ and ${}^\uvAlg V'_{\gen}=[\gen\xleftarrow{v}\gen]\oplus[\gen])$. This problem is mitigated by the fact that the curves associated with ${}^\mathcal{E}W$ and ${}^\mathcal{E}W'$ will differ only by how many times their ends wrap around the two punctures, and initially we regard such curves as the same. Another way to mitigate this problem is to find vertically and horizontally simplified bases  $\{\xi_i\},~\{\eta_j\}$ for ${}^\uvAlg V_{\gen}$ at the outset,  and apply the operation~\eqref{ping1} to the basis $\{\xi_i\}$ and  the operation~\eqref{ping2} to the basis $\{\eta_i\}$. This will ensure that the curve associated with ${}^\mathcal{E}W$ will not have extra wrapping around the punctures (and, of course, there may be non-trivial train tracks in the middle as in Figure~\ref{fig:run-exp-2}).
\end{aside}
\labellist 
\small
\pinlabel $v$ at 15 42  \pinlabel $\iota_{\gen}$ at 85 42
\pinlabel $a$ at 170 25 \pinlabel $b$ at 170 62 \pinlabel $\iota_{\gen}$ at 215 42
\pinlabel $\iota_\bullet$ at 142 10 \pinlabel $\iota_\circ$ at 142 80
\endlabellist
\begin{figure}[t]
	\centering
	\includegraphics[scale=1]{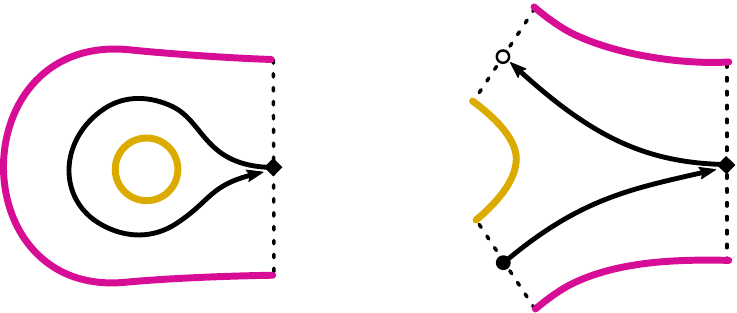}
	\caption{A quiver associated with the annulus describing the algebra $\k[v]$, and a quiver for an algebra associated with an additional cut.  
	}\label{fig:cut-annulus}
\end{figure}

\labellist %\small
\pinlabel {$\protect{\left(\begin{smallmatrix}0&1\\ 1&0\end{smallmatrix}\right)}$} at 17 -2
\pinlabel {$\protect{\left(\begin{smallmatrix}1&\lambda\\ 0&1\end{smallmatrix}\right)}$} at 65 -2
\pinlabel {$\protect{\left(\begin{smallmatrix}\lambda&0\\ 0&1\end{smallmatrix}\right)}$} at 113 -2
\tiny 
\pinlabel $\lambda$ at 59 17 
\pinlabel $\lambda$ at 106.5 29 
\pinlabel $\lambda$ at 59 57 \pinlabel $-\!\lambda$ at 72 56.5
\pinlabel $\lambda$ at 106.5 71.5  \pinlabel $\frac{1}{\lambda}$ at 117 60.5
\endlabellist
\parpic[r]{
 \begin{minipage}{60mm}
 % \centering
 \captionsetup{type=figure}
 \includegraphics[scale=1.25]{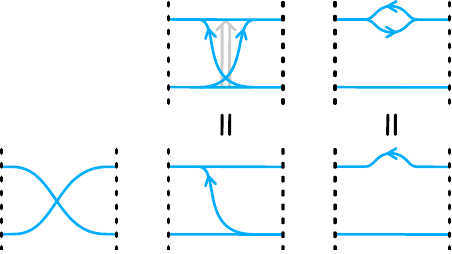}
  \vspace{0.15cm}
 \captionof{figure}{%
 	A crossing, a crossover switch, and a passing loop, each with the elementary matrix they represent by reading paths right-to-left. 
	To declutter pictures, we will be using the pictures in the lower row where the arrows pointing left-to-right are dropped.}\medskip
 \label{fig:anatomy}
  \end{minipage}
}

We now return to the main text and make some comments about our conventions, reviewing~\cite[Section~5.6]{KWZ}. The object appearing in the strip $\mathbf{a}\times[-1,1]$ represents an invertible matrix, where the $i^{\text{th}}$ column records the edges leaving the point labeled $x_i$ on $\mathbf{a}\times\{-1\}$ ($\mathbf{a}$ is oriented from top to bottom in our figures, so that $\{-1\}$ is the right most edge of the strip). Using the row-reduction algorithm, this matrix can be factorized into elementary matrices corresponding to three geometric sub-objects, as shown in Figure~\ref{fig:anatomy}. These sub-objects differ from the ones in~\cite{HRW}, where the coefficients are restricted to the two-element field. New in the context of general fields are the non-zero coefficients $\lambda\in \k$, recorded on the crossover switches (these correspond to crossover arrows from~\cite{HRW}), as well as the passing loops, which introduce coefficients at various points. The main point is that when two coefficients appear consecutively on one edge connecting the source and the target, the coefficients multiply, while if two edges share a common source and a common target, the coefficients on those edges add. We note that the geometric objects contain not only the information encoding $\varphi$ (reading right-to-left) but also the information about the inverse $\varphi^{-1}$ (reading left-to-right). As such, some of the data in the crossover switches and in the passing loops is superfluous. In particular, to simplify pictures below, we will record only the arrows running right-to-left. 

% Matt Hedden's comment: 
% (2) Provide more detailed explanation of the passage between elementary matrices and train tracks.  In particular, explain how the conventions expressed in Figures 7 and 9 are consistent.  The naïve reader will look at Figure 7, see the correspondence between matrices, and then look at the example preceding Figure 9 and think they do not follow the rules from Figure 7. 
%
% changes:
% * changed Figure~\ref{fig:anatomy} and its caption to make clear that we drop half of the arrows. 
% * make clear that the normal form is not a decomposition into elementary matrices, by changin the order in which we discuss it: First, we describe normal form in terms of matrices, then give an example, then give the geometric interpretation.

It is convenient to put the matrix representing $\varphi$ into a normal form, namely the {\bf LPU} normal form: 
Any invertible matrix can be written as a product of a {\bf L}ower triangular matrix, a {\bf P}ermutation matrix (which may be multiplied, additionally, by a diagonal matrix to change coefficients), and an {\bf U}pper triangular matrix.  
For example, the matrix $\left(\begin{smallmatrix}\lambda&1\\ 1&0\end{smallmatrix}\right)$ may be expressed as 
\[\left(\begin{smallmatrix}1&\lambda\\ 0&1\end{smallmatrix}\right)
\left(\begin{smallmatrix}0&1\\ 1&0\end{smallmatrix}\right)
=
\left(\begin{smallmatrix}1&0\\ \lambda^{-1}&1\end{smallmatrix}\right)
\left(\begin{smallmatrix}\lambda&0\\ 0&-\lambda^{-1}\end{smallmatrix}\right)
\left(\begin{smallmatrix}1&\lambda^{-1}\\ 0&1\end{smallmatrix}\right)\]
%This identity can be interpreted geometrically as follows:
and this identity has the following geometric interpretation: 
\[
\labellist 
\pinlabel $=$ at 78 15
\tiny
\pinlabel $\lambda$ at 11 15 
\pinlabel $\frac{1}{\lambda}$ at 101 21
\pinlabel $\lambda$ at 139 26 \pinlabel $-\frac{1}{\lambda}$ at 137 4
\pinlabel $\frac{1}{\lambda}$ at 167 13 
\endlabellist
\includegraphics[scale=1.1]{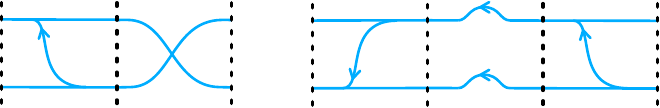}
\]
More generally, writing the matrix for $\varphi$ in {\bf LPU} normal form corresponds to modifying the train track in the region $\mathbf{a}\times[-1,1]$ such that the downward arrows are on the left, the upward arrows are on the right, and there is a permutation in the middle. A complete list of geometric moves corresponding to different factorizations into elementary matrices is given in~\cite[Figure~23]{KWZ}. As an example, the reader should compare Figures~\ref{fig:run-exp-2} and~\ref{fig:run-exp-3}.

\labellist \endlabellist
\parpic[r]{
 \begin{minipage}{80mm}
 \centering
 \bigskip
 \captionsetup{type=figure}
 \includegraphics[scale=1.1]{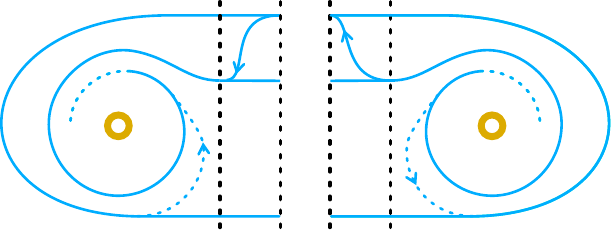}
 \captionof{figure}{Arrows running counterclockwise can be removed. }
 \label{fig:cancel}
  \end{minipage}%
  }
The reason this form is useful is that it allows us to remove arrows and simplify. This is possible in general, by appealing to an algorithm given in~\cite{HRW}, and ultimately gives rise to the proof of Theorem~\ref{thm:classification}; see~\cite[Section~5]{KWZ} for details. The main point is that arrows winding counterclockwise around a puncture can be removed. Namely, suppose there is an arrow near an edge of the strip $\strip$ that, when pushed into the relevant annulus, runs counterclockwise between curve-segments with different amounts of wrapping. Then there is a homotopy equivalence that produces a new train track---with the counterclockwise arrow removed---representing the same type~D structure; see Figure~\ref{fig:cancel}. This is described in detail in~\cite[Lemma 5.11]{KWZ}. The result of this procedure, applied to the example described in Figure~\ref{fig:run-exp-3}, is shown in Figure~\ref{fig:run-exp-4}.

\labellist 
\pinlabel $\overbrace{\phantom{aaaaaaaaaaaaa}}$ at 138 138
\pinlabel {$\protect{
\left(\begin{smallmatrix}1&0&0\\ 0&1&0\\ 0&\!\!-\!1&1 \end{smallmatrix}\right)\!\!
\left(\begin{smallmatrix}0&0&1\\ \!\!\!-\!1&0&0\\0&1&0 \end{smallmatrix}\right)\!\!
\left(\begin{smallmatrix}1&\!\!-\!1&0\\ 0&1&0\\0&0&1 \end{smallmatrix}\right)}
$} at 138 154
%\pinlabel {$\protect{\left(\begin{smallmatrix}0&0&1\\-1&1&0\\1&0&0 \end{smallmatrix}\right)}$} at 300 22
\tiny
\pinlabel $-\!1$ at 119 94 
\pinlabel $-\!1$ at 142  122
\pinlabel $+\!1$ at 142  102.5
\pinlabel $-\!1$ at 160 110 
\endlabellist
\begin{figure}[p]
	\centering
\vspace{1cm}
 \includegraphics[scale=1.1]{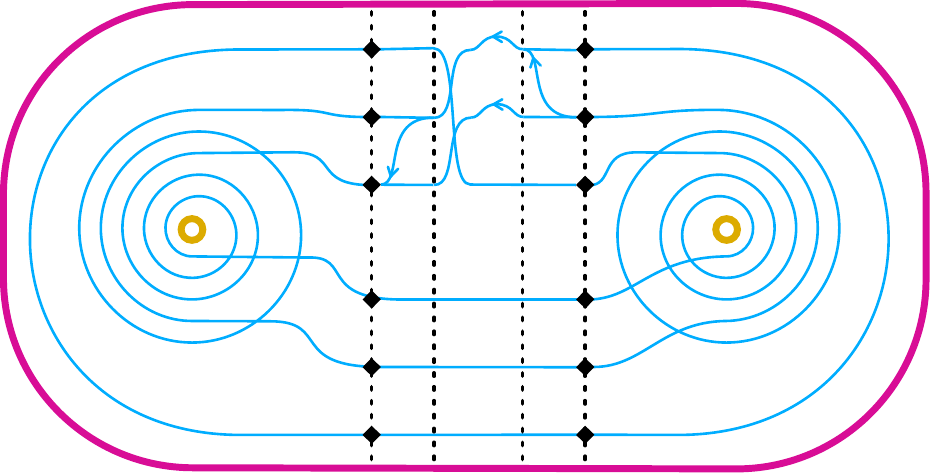}
% \vspace{0.5cm}
 \caption{Modifying the train track from Figure~\ref{fig:run-exp-2} according to an $\mathbf{LPU}$ decomposition of the matrix. }\label{fig:run-exp-3}
\end{figure}

\labellist 
\pinlabel $(-1)$ at 180 18
\endlabellist
\begin{figure}[p]
	\centering
 \includegraphics[scale=1.1]{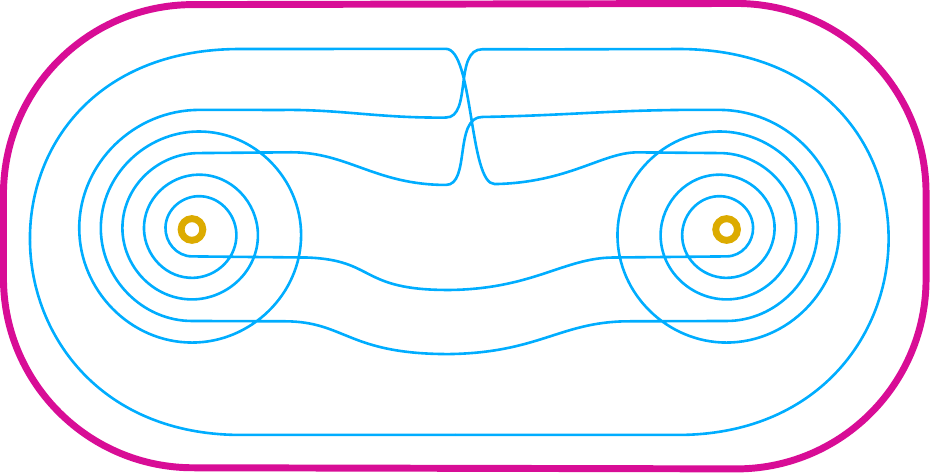}
% \vspace{0.5cm}
 \caption{Modifying the train track from Figure~\ref{fig:run-exp-3} by removing the counter-clockwise arrows. This produces an immersed curve---an object that is equivalent to the train track from Figure~\ref{fig:run-exp-1}, and which carries a 1-dimensional local system with automorphism that multiplies by $-1$.
 }\label{fig:run-exp-4}
\end{figure}

\labellist \small
\pinlabel {$\big(\k^2,\protect{
\left(\begin{smallmatrix}1&0\\ \lambda&1 \end{smallmatrix}\right)}\big)$} at 237 60.5
\tiny
\pinlabel $\lambda$ at 58 56
\endlabellist
\begin{figure}[p]
 \includegraphics[scale=1.1]{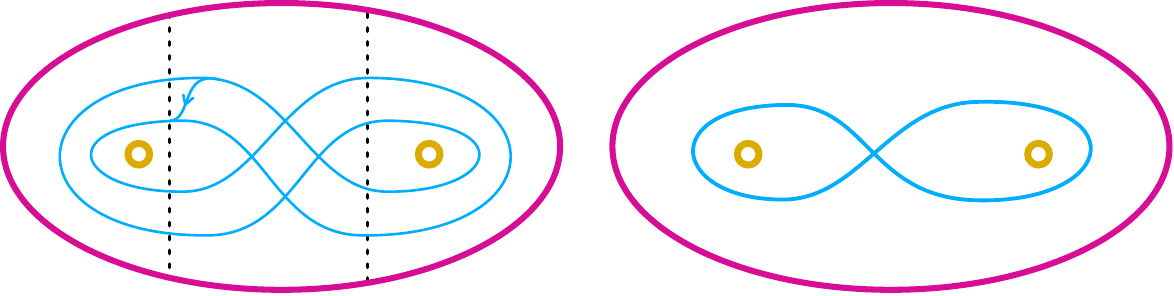}
 \caption{An arrow that cannot be cancelled gives rise to a non-trivial local system. }
 \label{fig:non-trivial-local}
\end{figure}

Recall that a local system over an immersed curve is a vector bundle over the curve. In general, all of our curves carry local systems, but when the associated bundle is one-dimensional and trivial we drop it from the notation. When working with signs, one-dimensional local systems are quite common as the coefficients along any given curve component multiply. Of course, non-compact curves do not carry interesting local systems since all vector bundles are trivial in this case. On the other hand, for compact curves it should be clear from the construction described above where a local system can arise: If two compact curves run parallel, then a crossover switch running between them  cannot be removed by a chain isomorphism of type~D structures. In general, local systems provide a clean way of presenting the relevant invariants, while the formalism expressing curves with local systems in terms of train tracks gives a concrete means of working with these objects. An example is shown in Figure~\ref{fig:non-trivial-local}; notice that, by replacing $\varphi$ with $\varphi^{-1}$ in this example, one can obtain a vertically simplified basis or a horizontally simplified basis, but not both simultaneously. It appears to still be an open question if such phenomena arises for invariants associated with knots; see~\cite[Remark 2.9]{Hom-thesis}. 

\labellist \small
\pinlabel $\rho_0$ at 27 10, \pinlabel $\rho_3$ at 149 10
\pinlabel $\rho_1$ at 27 110, \pinlabel $\rho_2$ at 149 110
\pinlabel $\mathbf{a}_\bullet$ at 95 25
\pinlabel $\mathbf{a}_\circ$ at 28 52
\endlabellist
\parpic[r]{
 \begin{minipage}{55mm}
 \vspace{60pt}
 \centering
 \captionsetup{type=figure}
 \includegraphics[scale=0.75]{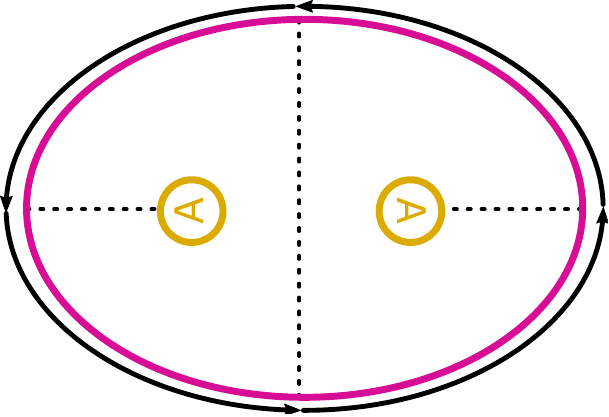}
 \captionof{figure}{An arc system for the extended algebra $\TsA$. The two discs are identified, producing a handle.}\label{fig:torus-with-quiver}
 \vspace{-40pt}
  \end{minipage}%
  }
\section{Adding a handle}
We now introduce the second algebra: the extended torus algebra $\TsA$. This algebra is introduced in~\cite{HRW}, and is also the algebra arising naturally in our setting.  
By construction, the map $\handle$ takes the twice-punctured disk to the once-punctured torus $T$.  An arc system for the latter is shown in Figure~\ref{fig:torus-with-quiver}, from which the associated quiver
\[
\setlength\mathsurround{0pt}
% to avoid rendering issue with gtpart.cls
% see https://tex.stackexchange.com/questions/217753/rendering-issue-with-tikz-cd-package-and-gtpart-cls-document-class
\begin{tikzcd}[column sep=2cm]
\bullet
\arrow[bend left=10]{r}[description]{\rho_{1}}
\arrow[bend right=30]{r}[description]{\rho_{3}}
&
\circ
\arrow[bend right=30]{l}[description]{\rho_{0}}
\arrow[bend left=10]{l}[description]{\rho_{2}}
\end{tikzcd}
\]
can be extracted---as before we contract the arcs to the quiver vertices.
Consulting Figure~\ref{fig:boundary}, note that $\mathbf{a}_\bullet$ is identified with the meridian $\mu$ and $\mathbf{a}_\circ$ is identified with the choice of longitude $\lambda$. With this arc system we associate an algebra $\TsA$. 
Analogous to the relation $uv=0$ from Figure~\ref{fig:disk-with-quiver}, the algebra $\TsA$ has relations 
$$\rho_{i+1}\rho_{i}=0$$  
(indices interpreted modulo 4), as explained in Section~\ref{sub:geo}. Note that the products $\rho_i\rho_{i+1}=\rho_{i(i+1)}$ are non-zero. For consistency with~\cite[Section~3.1]{HRW} we would need to add an additional relation $\rho_0\rho_1\rho_2\rho_3\rho_0=0$, but this is not necessary in the present setting.

The arc system associated with $\TsA$ decomposes the torus into a single disk, so type~D structures associated with compact train-tracks will be curved.  We fix the curvature term $c=\rho_{0123}+\rho_{1230}+\rho_{2301}+\rho_{3012}$. Recall that a curved type~D structure over $\TsA$ satisfies the compatibility condition 
\[(\mu \otimes \id_V)\circ ( \id_{\TsA} \otimes \, d ) \circ d = c\cdot\id_{\TsA}\]
and that, in this setting, the underlying $\k$-vector space decomposes so that $V=V_\bullet\oplus V_\circ$ as an $\sI$-module. 

The torus algebra is the quotient $\sA=\TsA/(\rho_0)$. Notice that in this quotient the curvature vanishes and the compatibility condition for type~D structures given in Section~\ref{sub:alg} is recovered. This algebra is explored in depth in~\cite[Section~11]{LOT-main} and in~\cite{HRW}. 

% Matt Hedden: (3) In sections 3 and 4, you outsource a number of details to [5] and [13], with specific references to sections therein.  I think adding some of these details in the text and explaining them in such a way to make it self-contained and accessible to someone coming at this subject for the first or second time would help.  You could keep the explicit references as parentheticals, but briefly explain the ideas underlying them and add some of the most easily digested details involved.

\section{The proof of Theorem~\ref{thm:LOT}}\label{sub:handle}
 
To set the stage, we first describe three general constructions. First, given a type~D structure ${}^{\k[u]}N$ over the polynomial ring $\k[u]$, there is a natural way to produce a dg module/chain complex over $\k[u]$: Substitute each generator $\gen$ in ${}^{\k[u]}N$ with a copy of the ring $\k[u]$, producing a free $\k[u]$-module, and then endow this module with a differential by substituting every arrow $\gen \xrightarrow{\ell u^n} \gen$  in ${}^{\k[u]}N$ with a map $\k[u] \xrightarrow{\cdot (\ell u^n)} \k[u]$ (where $\ell\in \k$). We denote the resulting dg module by $\k[u]\boxtimes {}^{\k[u]}N$, because it coincides with the result of box tensoring the type~D structure with the module $\k[u]$ viewed as a bimodule over itself~\cite[Section~2.3.2]{LOT-bim}. Note that this operation respects homotopy equivalences and also can be reversed~\cite[Proposition~2.3.18]{LOT-bim}, albeit in a less than straightforward way.

For the second construction, let $\k[v]$ be the graded polynomial ring in one variable with grading $a(v)=1$. (Below, $a$ will be the Alexander grading.)
Suppose ${}^{\k[v]}N=\bigoplus_{a\in\Z} {}^{\k[v]}N^a$ is a graded type~D structure over $\k[v]$ such that the differential preserves the grading $a$. 
We can then  produce a complex ${}^{\k[v]}N\big|_{v=1}$ by substituting arrows $\gen \xrightarrow{\ell v^n} \gen$ in ${}^{\k[v]}N$ by arrows $\gen \xrightarrow{\ell} \gen$. Clearly, this amounts to passing to the quotient $\k=\k[v]/(v-1)$. 
However, since $a(v)=1$, all the differentials in  ${}^{\k[v]}N$  that involved $v^n$, for $n\neq 0$, now change the grading in ${}^{\k[v]}N\big|_{v=1}$ by $n$. 
Thus, we can consider ${}^{\k[v]}N\big|_{v=1}$ as a filtered chain complex where the filtration levels are $\mathcal F_j=\oplus_{a\leq j} N^a$. As a category, type~D structures over $\k[v]$ are equivalent to filtered chain complexes via the construction above. In particular, type~D structure homomorphisms and homotopies between them precisely correspond to filtered chain maps and filtered homotopies between them.

The third construction is similar to the second. Given, a graded type~D structure  ${}^{\k[v]}N$ over $\k[v]$ whose differential preserves the grading $a$, we define a complex ${}^{\k[v]}N\big|_{v=0}$ by removing all arrows $\gen \xrightarrow{\ell v^n} \gen$, $n>0$, in ${}^{\k[v]}N$. This amounts to passing to the quotient $\k=\k[v]/(v)$, or equivalently, to passing to the associated graded complex of the filtered complex ${}^{\k[v]}N\big|_{v=1}$. 

We can now provide a dictionary between the knot Floer structures used here and those in~\cite{LOT-main}. In this paper, the most general knot Floer invariant is the type~D structure ${}^{\k[u,v]}\mathit{CFK}(S^3,K)$.  In~\cite{LOT-main}, two kinds of invariants appear. The first is the filtered chain complex $\mathit{CFK}^-(S^3,K)$ over $\k[u]$, which is a dg module over $\k[u]$ filtered with respect to the Alexander grading.
% (because the differentials corresponding to the $v$ variable do not pick up the $v$ variable anymore\comment{Claudius: This is very cryptic. What does "anymore" mean? Better just state the formula.}). Thus we need to turn the type~D structure ${}^{\k[u,v]}\mathit{CFK}(S^3,K)$ into both module and filtered complex using two different variables, and so we obtain the formula
It is obtained from ${}^{\k[u,v]}\mathit{CFK}(S^3,K)$ by applying the first construction to the variable $u$ and the second construction to the variable $v$:
$$\mathit{CFK}^-(S^3,K)= \k[u]\boxtimes{}^{\k[u]}\left({}^{\k[u,v]}\mathit{CFK}(S^3,K)\big|_{v=1}\right)$$
The second invariant used in~\cite{LOT-main} is $\mathit{gCFK}^-(S^3,K)$, the associated graded complex of $\mathit{CFK}^-(S^3,K)$. 
% It is a dg module over $\k[u]$ with the differentials corresponding to the $v$ variable vanishing. As such, in order to not count the $v$ differentials in the type~D structure ${}^{\k[u,v]}\mathit{CFK}(S^3,K)$ and also turn it into a module over $\k[u]$, we obtain
It is obtained from ${}^{\k[u,v]}\mathit{CFK}(S^3,K)$ by applying the first construction to the variable $u$ and the third construction to the variable $v$:
$$ \mathit{gCFK}^-(S^3,K)= \k[u]\boxtimes{}^{\k[u]}\left({}^{\k[u,v]}\mathit{CFK}(S^3,K)\big|_{v=0}\right) $$

\begin{example}
Consider the right-hand trefoil and its knot Floer invariants. The type~D structure invariant is as follows:
$${}^{\k[u,v]}\mathit{CFK}(S^3, \RHT)=[\gen^1_1 \xleftarrow{u} \gen^0_1 \xrightarrow{v}  \gen^{-1}_1] $$
where the superscripts and subscripts indicate the Alexander and $\delta$ gradings respectively. Recall that the Alexander and $\delta$ gradings are $\text{gr}(u)=(-1,1),~\text{gr}(v)=(1,1)$, so that the differential in the type~D structure is of bidegree $(a,\delta)=(0,1)$. The filtered chain complex over $\k[u]$ now becomes
$$\mathit{CFK}^-(S^3,K)= \k[u]\boxtimes{}^{\k[u]}\left({}^{\k[u,v]}\mathit{CFK}(S^3,K)\big|_{v=1}\right)=[\k[u]^1_1 \xleftarrow{\cdot u}\k[u]^0_1  \xrightarrow{1} \k[u]^{-1}_1]$$
while the associated graded chain complex over $\k[u]$ is equal to 
$$\mathit{gCFK}^-(S^3,K)= \k[u]\boxtimes{}^{\k[u]}\left({}^{\k[u,v]}\mathit{CFK}(S^3,K)\big|_{v=0}\right)=[\k[u]^1_1 \xleftarrow{\cdot u}\k[u]^0_1] \oplus [\k[u]^{-1}_1]$$
\end{example}

% \comment{Matt Hedden: (4) In section 4, provide more detail (maybe simply another figure) describing passage from $u^i$ or $v^i$ to the stable chains.  This was clear to me, but could be made easier with another figure.  Similarly provide more details on the unstable chain analysis and the reverse procedure.  As I think this section is where the most new content is present (is this right?), it warrants the most expansion of detail. }

We now proceed to the proof. We start with the knot Floer type~D structure ${}^\uvAlg\mathit{CFK}(S^3,K)={}^{\F[u,v]}\mathit{CFK}(S^3,K)\big|_{vu=0}$, and then homotope it to a representative (following the steps from Section~\ref{sub:geo}) from which the curve invariant $\bgamma$ can be extracted. With the dictionary above in mind, \cite[Theorem~A.11]{LOT-main} describes in detail how to pass from ${}^\uvAlg\mathit{CFK}(S^3,K)$ to the type~D structure  ${}^\sA\widehat{\mathit{CFD}}(M)$, which then produces a curve $\HFhat(M)$ in the punctured torus $\partial M$. Our task is to prove that the resulting curve coincides with  $\handle(\bgamma)$. 

We focus on segments of the curve $\bgamma$ in each of the annuli $A_v$ and $A_u$, and consider their images under the map $\handle$. Starting with an illustrative example, the image of a curve segment corresponding  to the arrow $\gen\xrightarrow{v^3}\gen$ is drawn in thick in Figure~\ref{fig:extra}, relative to the arc system of the algebra $\TsA$.
 \begin{figure}[ht]
 \begin{subfigure}[b]{0.48\textwidth}
 \centering
  \labellist \small
\pinlabel $\rho_0$ at 27 10, \pinlabel $\rho_3$ at 149 10
\pinlabel $\rho_1$ at 27 105, \pinlabel $\rho_2$ at 149 105
\endlabellist
 \includegraphics[scale=1]{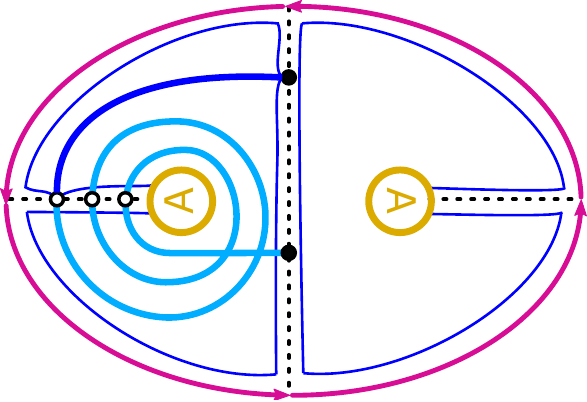}  
 \caption{Sample part of the curve corresponding to a stable chain from \cite[Theorem~A.11]{LOT-main}.}
 \label{fig:extra}
 \end{subfigure}
 \begin{subfigure}[b]{0.48\textwidth}
 \centering
 \labellist \small
\pinlabel $\rho_0$ at 27 10, \pinlabel $\rho_3$ at 149 10
\pinlabel $\rho_1$ at 27 105, \pinlabel $\rho_2$ at 149 105
\endlabellist
 \includegraphics[scale=1]{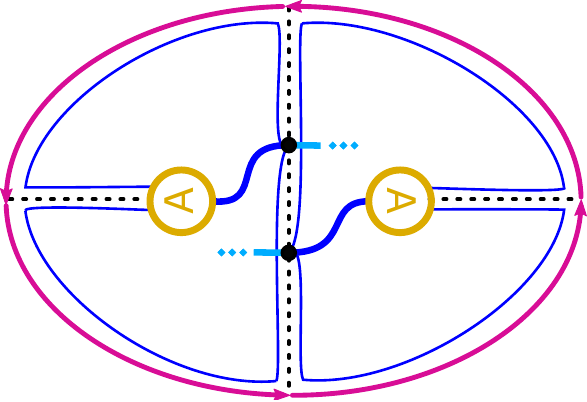}  
 \caption{Part of the curve corresponding to the unstable chain from \cite[Theorem~A.11]{LOT-main}.}
 \label{fig:extra2}
  \end{subfigure}
 \caption{}
  \label{fig:extra12}
  \end{figure}
Focusing on the first part of this segment, shown are the two ways in can be retracted to the boundary of the torus union the two arcs: In one case the homotoped path runs along $\rho_1$, and in the other case it runs along $\rho_2$ then   $\rho_3$  then  $\rho_0$. In the type~D structure language then, according to \cite{HRW} and the discussion in Section~\ref{sub:geo}, this part of the curve results in $
\setlength\mathsurround{0pt}
% to avoid rendering issue with gtpart.cls
% see https://tex.stackexchange.com/questions/217753/rendering-issue-with-tikz-cd-package-and-gtpart-cls-document-class
\begin{tikzcd}[column sep=1cm]
\bullet
\arrow[bend left=15]{r}[description]{\rho_{1}}
&
\circ
\arrow[bend left=15]{l}[description]{\rho_{230}}
\end{tikzcd}$. Similarly, the whole thick curve segment depicted in Figure~\ref{fig:extra} corresponds to the following part of a type~D structure over $\TsA$: 
\[
\begin{tikzpicture}[>=latex, scale=0.9] 
    \node at (0,0) {$\bullet$}; 
    %\node at (1,0) {$\circ$}; 
    \node at (2,0) {$\circ$};
    %\node at (3,0) {$\circ$};
    \node at (4,0) {$\circ$};
    %\node at (5,0) {$\bullet$}; 
    \node at (6,0) {$\circ$};
    \node at (8,0) {$\bullet$};
\draw[thick,->,shorten >=0.1cm, shorten <=0.1cm,] (0,0) to [out=20,in=160, looseness=1] (2,0); \node at (1,0.4) {\scriptsize${\rho_{1}}$};
\draw[thick,<-,shorten >=0.1cm, shorten <=0.1cm,] (0,0) to [out=-20,in=-160, looseness=1] (2,0); \node at (1,-0.4) {\scriptsize${\rho_{230}}$};
\draw[thick,->,shorten >=0.1cm, shorten <=0.1cm,] (2,0) to [out=20,in=160, looseness=1] (4,0); \node at (3,0.4) {\scriptsize${\rho_{01}}$};
\draw[thick,<-,shorten >=0.1cm, shorten <=0.1cm,] (2,0) to [out=-20,in=-160, looseness=1] (4,0); \node at (3,-0.4) {\scriptsize${\rho_{23}}$};
\draw[thick,->,shorten >=0.1cm, shorten <=0.1cm,] (4,0) to [out=20,in=160, looseness=1] (6,0); \node at (5,0.4) {\scriptsize${\rho_{01}}$};
\draw[thick,<-,shorten >=0.1cm, shorten <=0.1cm,] (4,0) to [out=-20,in=-160, looseness=1] (6,0); \node at (5,-0.4) {\scriptsize${\rho_{23}}$};
 \draw[thick,->,shorten >=0.1cm, shorten <=0.1cm,] (6,0) to [out=20,in=160, looseness=1] (8,0); \node at (7,0.4) {\scriptsize${\rho_{0}}$};
 \draw[thick,<-,shorten >=0.1cm, shorten <=0.1cm,] (6,0) to [out=-20,in=-160, looseness=1] (8,0); \node at (7,-0.4) {\scriptsize${\rho_{123}}$};
  \end{tikzpicture}
  \]
More generally, the image of a curve segment corresponding  to the arrow $\gen\xrightarrow{v^i}\gen$ is 
%image of $v^i$
\[
\begin{tikzpicture}[>=latex, scale=0.9] 
		\node at (0,0) {$\bullet$}; 
		%\node at (1,0) {$\circ$}; 
		\node at (2,0) {$\circ$};
		%\node at (3,0) {$\circ$};
		\node at (4,0) {$\circ$};
		%\node at (5,0) {$\bullet$}; 
		\node at (6,0) {$\circ$};
		\node at (8,0) {$\circ$};
		\node at (10,0) {$\bullet$};
\draw[thick,->,shorten >=0.1cm, shorten <=0.1cm,] (0,0) to [out=20,in=160, looseness=1] (2,0); \node at (1,0.4) {\scriptsize${\rho_{1}}$};
\draw[thick,<-,shorten >=0.1cm, shorten <=0.1cm,] (0,0) to [out=-20,in=-160, looseness=1] (2,0); \node at (1,-0.4) {\scriptsize${\rho_{230}}$};
\draw[thick,->,shorten >=0.1cm, shorten <=0.1cm,] (2,0) to [out=20,in=160, looseness=1] (4,0); \node at (3,0.4) {\scriptsize${\rho_{01}}$};
\draw[thick,<-,shorten >=0.1cm, shorten <=0.1cm,] (2,0) to [out=-20,in=-160, looseness=1] (4,0); \node at (3,-0.4) {\scriptsize${\rho_{23}}$};
\draw[thick,->,shorten >=0.1cm, shorten <=0.1cm,] (4,0) to [out=20,in=160, looseness=1] (6,0); \node at (5,0.4) {\scriptsize${\rho_{01}}$};
\draw[thick,<-,shorten >=0.1cm, shorten <=0.1cm,] (4,0) to [out=-20,in=-160, looseness=1] (6,0); \node at (5,-0.4) {\scriptsize${\rho_{23}}$};
 \node at (7,0) {\scriptsize${\cdots}$};
 \draw[thick,->,shorten >=0.1cm, shorten <=0.1cm,] (8,0) to [out=20,in=160, looseness=1] (10,0); \node at (9,0.4) {\scriptsize${\rho_{0}}$};
 \draw[thick,<-,shorten >=0.1cm, shorten <=0.1cm,] (8,0) to [out=-20,in=-160, looseness=1] (10,0); \node at (9,-0.4) {\scriptsize${\rho_{123}}$};
\node at (5,0.60) {$\overbrace{\phantom{aaaaaaaaaaaaaaaaaaaaaaaaaaaaaaaaa}}$};
\node at (5,1)  {$\dim V_\circ=i$} ;
	\end{tikzpicture}
	\]
Analogously, the image of a curve segment corresponding to the arrow $\gen\xrightarrow{u^i}\gen$ is 
%image of $u^i$
\[
\begin{tikzpicture}[>=latex, scale=0.9] 
		\node at (0,0) {$\bullet$}; 
		%\node at (1,0) {$\circ$}; 
		\node at (2,0) {$\circ$};
		%\node at (3,0) {$\circ$};
		\node at (4,0) {$\circ$};
		%\node at (5,0) {$\bullet$}; 
		\node at (6,0) {$\circ$};
		\node at (8,0) {$\circ$};
		\node at (10,0) {$\bullet$};
\draw[thick,->,shorten >=0.1cm, shorten <=0.1cm,] (0,0) to [out=20,in=160, looseness=1] (2,0); \node at (1,0.4) {\scriptsize${\rho_{3}}$};
\draw[thick,<-,shorten >=0.1cm, shorten <=0.1cm,] (0,0) to [out=-20,in=-160, looseness=1] (2,0); \node at (1,-0.4) {\scriptsize${\rho_{012}}$};
\draw[thick,->,shorten >=0.1cm, shorten <=0.1cm,] (2,0) to [out=20,in=160, looseness=1] (4,0); \node at (3,0.4) {\scriptsize${\rho_{23}}$};
\draw[thick,<-,shorten >=0.1cm, shorten <=0.1cm,] (2,0) to [out=-20,in=-160, looseness=1] (4,0); \node at (3,-0.4) {\scriptsize${\rho_{01}}$};
\draw[thick,->,shorten >=0.1cm, shorten <=0.1cm,] (4,0) to [out=20,in=160, looseness=1] (6,0); \node at (5,0.4) {\scriptsize${\rho_{23}}$};
\draw[thick,<-,shorten >=0.1cm, shorten <=0.1cm,] (4,0) to [out=-20,in=-160, looseness=1] (6,0); \node at (5,-0.4) {\scriptsize${\rho_{01}}$};
 \node at (7,0) {\scriptsize${\cdots}$};
 \draw[thick,->,shorten >=0.1cm, shorten <=0.1cm,] (8,0) to [out=20,in=160, looseness=1] (10,0); \node at (9,0.4) {\scriptsize${\rho_{2}}$};
 \draw[thick,<-,shorten >=0.1cm, shorten <=0.1cm,] (8,0) to [out=-20,in=-160, looseness=1] (10,0); \node at (9,-0.4) {\scriptsize${\rho_{301}}$};
\node at (5,0.60) {$\overbrace{\phantom{aaaaaaaaaaaaaaaaaaaaaaaaaaaaaaaaa}}$};
\node at (5,1)  {$\dim V_\circ=i$} ;
	\end{tikzpicture}
	\]
Passing to the quotient algebra $\Alg$ by setting $\rho_0=0$ simplifies the above two images to 
\[
\begin{tikzpicture}[>=latex, scale=0.9] 
		\node at (0,0) {$\bullet$}; 
		%\node at (1,0) {$\circ$}; 
		\node at (2,0) {$\circ$};
		%\node at (3,0) {$\circ$};
		\node at (4,0) {$\circ$};
		%\node at (5,0) {$\bullet$}; 
		\node at (6,0) {$\circ$};
		\node at (8,0) {$\circ$};
		\node at (10,0) {$\bullet$};
\draw[thick,->,shorten >=0.1cm, shorten <=0.1cm,] (0,0) to (2,0); \node at (1,-0.2) {\scriptsize${\rho_{1}}$};
\draw[thick,<-,shorten >=0.1cm, shorten <=0.1cm,] (2,0) to (4,0); \node at (3,-0.2) {\scriptsize${\rho_{23}}$};
\draw[thick,<-,shorten >=0.1cm, shorten <=0.1cm,] (4,0) to  (6,0); \node at (5,-0.2) {\scriptsize${\rho_{23}}$};
 \node at (7,0) {\scriptsize${\cdots}$};
 \draw[thick,<-,shorten >=0.1cm, shorten <=0.1cm,] (8,0) to  (10,0); \node at (9,-0.2) {\scriptsize${\rho_{123}}$};
\node at (5,0.25) {$\overbrace{\phantom{aaaaaaaaaaaaaaaaaaaaaaaaaaaaaaaaa}}$};
\node at (5,0.65)  {$\dim V_\circ=i$} ;
	\end{tikzpicture}
	\]
and
\[
\begin{tikzpicture}[>=latex, scale=0.9] 
		\node at (0,0) {$\bullet$}; 
		%\node at (1,0) {$\circ$}; 
		\node at (2,0) {$\circ$};
		%\node at (3,0) {$\circ$};
		\node at (4,0) {$\circ$};
		%\node at (5,0) {$\bullet$}; 
		\node at (6,0) {$\circ$};
		\node at (8,0) {$\circ$};
		\node at (10,0) {$\bullet$};
\draw[thick,->,shorten >=0.1cm, shorten <=0.1cm,] (0,0) to (2,0); \node at (1,-0.2) {\scriptsize${\rho_{3}}$};
\draw[thick,->,shorten >=0.1cm, shorten <=0.1cm,] (2,0) to (4,0); \node at (3,-0.2) {\scriptsize${\rho_{23}}$};
\draw[thick,->,shorten >=0.1cm, shorten <=0.1cm,] (4,0) to  (6,0); \node at (5,-0.2) {\scriptsize${\rho_{23}}$};
 \node at (7,0) {\scriptsize${\cdots}$};
 \draw[thick,->,shorten >=0.1cm, shorten <=0.1cm,] (8,0) to  (10,0); \node at (9,-0.2) {\scriptsize${\rho_{2}}$};
\node at (5,0.25) {$\overbrace{\phantom{aaaaaaaaaaaaaaaaaaaaaaaaaaaaaaaaa}}$};
\node at (5,0.65)  {$\dim V_\circ=i$} ;
	\end{tikzpicture}
	\]
These are precisely the two {\em stable} chains appearing in the statement of~\cite[Theorem~A.11]{LOT-main}: According to their result, these are the parts of ${}^\sA\widehat{\mathit{CFD}}(M)$ that correspond to the differentials $\gen\xrightarrow{v^i}\gen$ and $\gen\xrightarrow{u^i}\gen$ in ${}^\uvAlg\mathit{CFK}(S^3,K)$.

\labellist \tiny
\pinlabel $w$ at 26 76
\pinlabel $z$ at 27 65
\small
\pinlabel $\rho_1$ at 51 110
\pinlabel $\rho_0$ at 59 84
\pinlabel $\rho_3$ at 59 40
\pinlabel $\rho_2$ at 51 4
\pinlabel $v$ at 86 83
\pinlabel $u$ at 86 30
\pinlabel $\mu$ at -4.5 58 \pinlabel $\lambda$ at 181 58
\endlabellist
\parpic[r]{
 \begin{minipage}{67mm}
 \centering
 \captionsetup{type=figure}
 \includegraphics[scale=0.9]{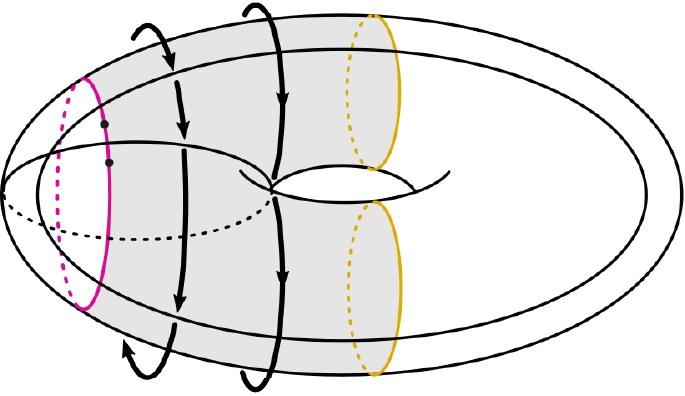}
 \captionof{figure}{Both algebras $\uvAlg$ and $\TsA$ in context. 
 }
 \label{fig:boundary}
  \end{minipage}%
  }
The main subtlety is the appearance of the {\em unstable} chain, which we have already touched on. Defining $\handle$ in such a way that there is no extra twisting introduced (see the left of Figure~\ref{fig:framings}), the straight segment running over the handle in Figure~\ref{fig:extra2} retracts in two ways shown, producing the final part of the type~D structure:
$
\setlength\mathsurround{0pt}
% to avoid rendering issue with gtpart.cls
% see https://tex.stackexchange.com/questions/217753/rendering-issue-with-tikz-cd-package-and-gtpart-cls-document-class
\begin{tikzcd}[column sep=1cm]
\bullet
\arrow[bend left=15]{r}[description]{\rho_{12}}
&
\bullet
\arrow[bend left=15]{l}[description]{\rho_{30}}
\end{tikzcd}$. Setting $\rho_0=0$ results in $\bullet \xrightarrow{\rho_{12}}\bullet$, which is precisely the unstable chain from~\cite[Theorem~A.11]{LOT-main}: According to their result, this is the final piece (in addition to the stable chains) in ${}^\sA\widehat{\mathit{CFD}}(M)$ (computed relative to the parameterization $(\mu,2\tau)$ of the torus $T^2=\partial M$). In~\cite[Theorem~A.11]{LOT-main}, this final piece connects the distinguished generators $\xi_0$ and $\eta_0$ in the vertically and horizontally simplified bases of $\mathit{CFK}^-(K)$. It is left to note that the two generators in Figure~\ref{fig:extra2} are precisely $\xi_0$ and $\eta_0$, because each is incident to only one arrow $\xrightarrow{v^i}$ or $\xrightarrow{u^j}$ in the complex ${}^\uvAlg\mathit{CFK}(S^3,K)$. We also remark that, while the unstable chain $\bullet \xrightarrow{\rho_{12}}\bullet$ corresponds to the $2\tau$-framing of the knot $K$, there are other type~D structure presentations of the unstable chain in \cite[Theorem~A.11]{LOT-main}, and those would correspond to other choices of twisting in $\handle$.

% In conclusion, we proved that the result of passing from the curve invariant $\bgamma$ to ${}^\uvAlg\mathit{CFK}(S^3,K)$, and then to the type~D structure  ${}^\sA\widehat{\mathit{CFD}}(M)$ via \cite[Theorem~A.11]{LOT-main} (picking the $(\mu,2\tau)$ framing for the bordering of $\partial M$), and then producing a curve $\HFhat(M)$ in the punctured torus via \cite{HRW}, coincides with the image $\handle(\bgamma)$, proving the first part of the theorem. 

The statement about the reverse operation follows from the discussion above. Namely, its clear that $\cut(\handle( \bgamma))=\bgamma$, and since we proved $\handle( \bgamma)=\HFhat(M)$,  we obtain $\cut(\HFhat(M))=\bgamma$.

\section{Comments on generalizations and related work}\label{sub:discussion}  
Perhaps the most interesting step in this constructive review of the Lipshitz--Ozsv\'ath--Thurston correspondence comes about when the endpoints of the non-compact component $\bgamma_0 \subset \bgamma$ are identified to give a new compact component in the once-punctured torus. Note that the output of $\handle$ is always a compact curve, and this is consistent with the observation that $\HFhat(M)$ is a compact curve. The latter, in turn, follows from the fact that $\CFD(M)$ is an extendable type~D structure~\cite[Appendix~A]{HRW}.

Joining the endpoints of the immersed curve $\bgamma_0$ associated with a knot $K$ requires a choice of automorphism of $\k^n$ where $n$  is the number of components in $\bgamma_0$. Denote the horizontal homology
$H^\mathbf{h} = H_*(C^{\mathbf{h}}\big|_{u=1})$ and the vertical homology $H^\mathbf{v} = H_*(C^{\mathbf{v}}\big|_{v=1}) $. Then in
Theorem~\ref{thm:LOT}, because the knot is in $S^3$, it follows that $n=1$ and the automorphism is given, tautologically, by \begin{equation*} H^\mathbf{v}(\mathit{CFK}^-(S^3,K))\cong H^\mathbf{h}(\mathit{CFK}^-(S^3,K))\cong\widehat{\mathit{HF}}(S^3)\cong\k\end{equation*} as explained in~\cite[Section~11.5]{LOT-main}. 
Thus, the operation $\handle$ is defined over any field, provided that we choose a coefficient $a\in \k$ when we identify the ends of $\bgamma_0$ along a handle. We choose this coefficient to be $+1$ so that the bordered invariant for the solid torus is a circle with the trivial local system. We note that bordered Floer homology is only defined over the two-element field $\mathbb{F}$. As such, the map $\handle$ and Theorem~\ref{thm:LOT} gives a {\em candidate} bordered invariant for the knot exterior when $\k\ne\mathbb{F}$. 

We now consider the general case of a knot $K$ in $Y$. Decomposing along spin$^c$-structures, the same strategy as above works if $Y$ is an L-space~\cite{HeddenLevine2016}.  More generally, however, one needs to know the isomorphism  \begin{equation*}H^\mathbf{v}(\mathit{CFK}^-(Y,K))\cong H^\mathbf{h}(\mathit{CFK}^-(Y,K))\cong\widehat{\mathit{HF}}(Y)\end{equation*} 
\labellist \small
\pinlabel $(\mathbf{k}^{n},\psi)$ at 68 50
\endlabellist
\parpic[r]{
 \begin{minipage}{68mm}
 \centering\vspace{15pt}
 \captionsetup{type=figure}
 \includegraphics[width=\textwidth]{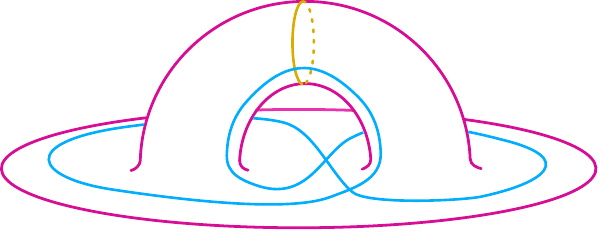}
 \captionof{figure}{A sample hypothetical local system, where $n=\dim\HFhat(Y)$. \\\hspace{\textwidth}  }\label{fig:loc-hyp}
  \end{minipage}%
  }
(which may be block-decomposed according to spin$^c$-structures). 
This recovers a generalization of~\cite[Theorem~A.11]{LOT-main}, which may be found in forthcoming work of Hockenhull~\cite{Hockenhull-prep} building on his invariant $\operatorname{Poly}(L,\Lambda)$~\cite{Hockenhull}.  From our perspective, the passage from the knot Floer homology of a knot $K$ in $Y$ to the bordered invariants of $Y\smallsetminus\openU(K)$ requires the isomorphism shown above. As there is a decomposition according to spin$^c$ structures, there is no loss of generality in considering the case where $Y$ is an integer homology sphere. When such a $Y$ is not an L-space, we have that $\dim\HFhat(Y)>1$ and, in principle, the automorphism $\psi$ induced by the isomorphism between the homologies $H^{\mathbf h}$ and $H^{\mathbf v}$ can be interesting. In particular, while all components of $\bgamma_0(K)$ carry trivial local systems, the new compact object $\handle(\bgamma_0(K))$ obtains an additional local system $(\mathbf{k}^{n},\psi)$; see Figure~\ref{fig:loc-hyp}. The key point of difference is that the output will be equivalent to a simply faced precurve (in the torus) in general, and a further application of the arrow sliding algorithm may be required to obtain immersed curves. The algebraic side of this story is laid out carefully by Hockenhull~\cite{Hockenhull-prep,Hockenhull}. 

Finally, Hanselman gives another approach~\cite{hanselman}: His construction takes the complex $\mathit{CFK}^-(K)$ and outputs an immersed curve in the strip covering the twice-punctured disk $D$, containing a countable set of pairs of punctures. This cover of the disk is useful for recording the Alexander grading, and also works with general fields (hence producing candidate bordered invariants). We advertise that Hanselman's construction has a different aim in mind, namely, a candidate bordered-minus invariant obtained by promoting the curves to describe type~D structures over $\k[u,v]$. 

\section{The Proof of Theorem~\ref{thm:connected_sum}}\label{sub:sum}
For simplicity we first focus on the case of two-element field $\k=\F$. A few properties of invariant $^{\uvAlg}\mathit{CFK}(S^3,K)$ are needed for the proof. First, given two type~D structures over the polynomial algebra $\k[u,v]$ or its quotient $\uvAlg$, their tensor product is another type~D structure:
\[
(V,d) \otimes (V',d') = (V\otimes_{\k} V',  d\otimes \id + \id \otimes d')
\] 
Now, reformulating~\cite[Theorem~7.1]{OS-hfk-0}, the behaviour of knot Floer homology under taking the connected sum can be described as follows:
$${}^{\uvAlg}\mathit{CFK}(S^3,K\# K') \simeq {}^{\uvAlg}({}^{\uvAlg}\mathit{CFK}(S^3,K)\otimes {}^{\uvAlg}\mathit{CFK}(S^3,K'))$$
The mirroring operation is also well understood~\cite[Proposition~3.7]{OS-hfk-0}: 
$${}^{\uvAlg}\mathit{CFK}(S^3,\text{m}  K) \simeq {}^{\uvAlg}\overline{\mathit{CFK}(S^3, K)}$$
where the latter is the \emph{dual type~D structure}, equal to the original one but with all differentials reversed~\cite[Definition~2.5]{LOT-mor} (since $\uvAlg$ is commutative, the fact that dualizing turns left type~D structure to right ones is not a problem). Finally, we need an algebraic relationship between morphism spaces of type~D structures~\cite[Section~2.2.3]{LOT-bim} and their tensor products. Given any two type~D structures, the definitions imply the following isomorphism of chain complexes:
$$  {\uvAlg} \boxtimes {}^{\uvAlg}({}^{\uvAlg}\overline{N}\otimes {}^{\uvAlg}{N'}) \cong \text{Mor}({}^{\uvAlg}{N}, {}^{\uvAlg}{N'})$$

With the properties above in place, the proof of Theorem~\ref{thm:connected_sum} is a sequence of isomorphisms:
\begin{align*} {}_{\uvAlg}\mathit{HFK}(S^3, \text{m} K \# K') 
&\cong  H_*(\uvAlg\boxtimes {}^{\uvAlg}\mathit{CFK}(S^3, \text{m} K \# K') ) \\
&\cong H_*(\uvAlg\boxtimes {}^{\uvAlg} [{}^{\uvAlg}\mathit{CFK}(S^3, \text{m} K) \otimes {}^{\uvAlg}\mathit{CFK}(S^3, K')] )  \\
&\cong H_*(\uvAlg\boxtimes {}^{\uvAlg} [{}^{\uvAlg} \overline{\mathit{CFK}(S^3, K)} \otimes {}^{\uvAlg} \mathit{CFK}(S^3, K') ] )  \\
&\cong H_*( \text{Mor} ({}^{\uvAlg} \mathit{CFK}(S^3, K), {}^{\uvAlg} \mathit{CFK}(S^3, K') ) )  \\
&\cong \mathit{HF} (\bgamma(K), \bgamma(K')) 
\end{align*} 
where the final isomorphism follows from the general description of morphism spaces between type~D structures over surface algebras~\cite[Theorem~1.5]{KWZ}.

 \begin{figure}[t]
 	\centering
\labellist \small
\pinlabel $\bgamma(U)$ at 173 89
\pinlabel $\bgamma(T_{2,3})$ at 33 27
\endlabellist
 \includegraphics[scale=0.9]{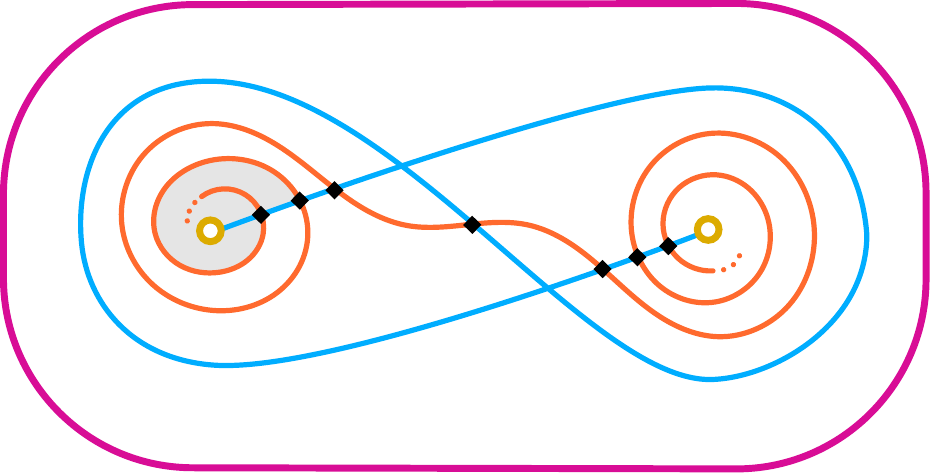}
 \caption{Illustrating Theorem~\ref{thm:connected_sum} in the case of the unknot $K=U$ and the right-hand trefoil $K'=T_{2,3}$. The curve $\bgamma(U)$ is a horizontal arc connecting the punctures, but because we are in the wrapped setting, one needs to wrap $\bgamma(U)$ infinitely many times around the punctures when pairing with another curve.
 }\label{fig:pairing}
\end{figure}

The recipe for adding signs follows the Koszul sign rule, which is discussed in~\cite[Section~12]{OS-bordered-matchings} in detail. We find that the resulting signs are a bit more more natural if one considers right type~D structures~\cite[Example~2.10]{KWZ}, rather than left ones~\cite[Section~12.3]{OS-bordered-matchings}, as then there are no extra signs when box tensoring with $-\boxtimes {}_\uvAlg\uvAlg_{\uvAlg}$; this is explained in~\cite[Page~19]{KWZ}. Now, since the algebra $\uvAlg$ is commutative, our left type~D structures can be viewed as right type~D structures, and after that filling in the signs becomes straightforward. We refer the reader to~\cite[Sections~2 and~5]{KWZ}.
\qed

To illustrate this gluing result, suppose $K=U$ and $K'=T_{2,3}$.  Then the knot Floer homology of the connected sum is equal to 
\begin{align*}
{}_{\uvAlg}\mathit{HFK}(S^3,T_{2,3})
&=
H_*(\uvAlg \xleftarrow{\cdot v} \uvAlg \xrightarrow{\cdot u} \uvAlg)
\\
&=
[\cdots \xleftarrow{v}\gen \xleftarrow{v}\gen \xleftarrow{v}\gen \xrightarrow{u}\gen \xleftarrow{v} \gen  \xrightarrow{u}\gen \xrightarrow{u} \gen \xrightarrow{u} \cdots]
\end{align*}

where the arrows indicate the  $\uvAlg$-action. The corresponding wrapped Lagrangian Floer homology $ \mathit{HF} (\bgamma(U), \bgamma(T_{2,3}))$ is illustrated in Figure~\ref{fig:pairing}. Note that in this example the $\uvAlg$-action can be seen geometrically by counting Maslov index $2$ disks covering the punctures; one of these is shaded in the picture. The same is true for $\k[H]$-action on Bar-Natan homology, viewed as wrapped Lagrangian Floer homology of immersed curves in~\cite[Example~7.7]{KWZ}. In general, to recover these module-structures, only some of the Maslov index $2$ disks should be counted---we will investigate this in future work.

\section*{Acknowledgements} 
This short paper benefited from stimulating conversations with Jonathan Hanselman, Thomas Hockenhull, and Matthew Stoffregen. The work was initiated during the CRM50 thematic program {\it low dimensional 
topology} and we thank the CRM and CIRGET for hosting a great event. 

AK is supported by an AMS-Simons travel grant. 
LW is supported by an NSERC discovery/accelerator grant and was partially supported by funding from the Simons Foundation and the Centre de Recherches Math\'ematiques, through the Simons-CRM scholar-in-residence program.
CZ is supported by the Emmy Noether Programme of the DFG, Project number 412851057, and the SFB 1085 Higher Invariants in Regensburg.

\bibliographystyle{gtart}
\bibliography{references}

\end{document}